\documentclass[12pt]{amsart}
\usepackage{amsmath,amsthm,amscd,amssymb,bbm,graphicx}
\usepackage{floatflt,setspace, wrapfig}
\usepackage{graphics,multicol}%changepage}
\usepackage{mathrsfs} %gives fance script A font for attractor
\usepackage{epstopdf}
\usepackage{geometry}

\newtheorem{thrm}{Theorem}[section]
\newtheorem*{MT1}{Theorem 1}
\newtheorem*{MT2}{Theorem 2}
\newtheorem*{MT3}{Theorem 3}

\newtheorem{prop}{Proposition}[section]
\newtheorem{lem}{Lemma}[section]

\newcommand{\maps}{\rightarrow}
\newcommand{\nats}{\mathbb{N}_0}
\newcommand{\ind}{\mathbbm{1}}
\newcommand{\ev}{\mathbb{E}}
\newcommand{\ball}{B_{\rho}}

\providecommand{\abs}[1]{\lvert \, #1 \, \rvert}
\providecommand{\Abs}[1]{\biggl\lvert \, #1 \, \biggr \rvert}
\providecommand{\norm}[1]{\lVert \, #1 \, \rVert}
\providecommand{\floor}[1]{\left\lfloor \, #1 \, \right\rfloor}
\providecommand{\roof}[1]{\left\lceil \, #1 \, \right\rceil}

\newcommand{\bea}[1]{\begin{eqnarray}\label{#1}}
\newcommand{\eea}{\end{eqnarray}}

\title[Limiting distribution for balls]{Limiting distribution and error terms for the number of visits to balls
 in non-uniformly hyperbolic dynamical systems}
 \date{\today}
 
\begin{document}
 \maketitle
\authors{N Haydn\footnote{Department of Mathematics, University of Southern California,
Los Angeles, 90089-2532. E-mail: {\tt \email{nhaydn@usc.edu}}.},
 K Wasilewska\footnote{Department of Mathematics, Harvard-Westlake School,
Studio City, CA 91604. E-mail: {\tt \email{kwilliams@hw.com}}.}}

%\tableofcontents

%%%%%%%%%%%%%%%%%%%%%%%%%%%%%%%%%%%%%%%%%%%%
%%%%%%%%%%%%%%%%%%%%%%%%%%%%%%%%%%%%%%%%%%%%
%%%%%%%%%%%%%%%%%%        ABSTRACT
\begin{abstract}
We show that for systems that allow a Young tower construction with polynomially
decaying correlations the return times to metric balls are in the limit Poisson distributed.
We also provide error terms which are powers of logarithm of the radius. In order
to get those uniform rates of convergence the balls centres have to avoid a set whose
size is estimated to be of similar order. This result can be applied to non-uniformly
hyperbolic maps and to any invariant measure that satisfies a weak regularity
condition. In particular it shows that the return times to balls is Poissonian
for SRB measures on attractors.
\end{abstract}

%%%%%%%%%%%%%%%%%%%%%%%%%%%%%%%%%%%%%%%%%%%%
%%%%%%%%%%%%%%%%%%%%%%%%%%%%%%%%%%%%%%%%%%%%
%%%%%%%%%%%%%%%%%%             INTRODUCTION
\section{Introduction}

Poincar\'e's recurrence theorem~\cite{P99} established that for measure preserving maps points
return to neighbourhoods arbitrarily often almost surely. The return time for the first return
was quantified by Kac~\cite{K47} in 1947 and since then there have been efforts to describe the
return statistics for shrinking neighbourhoods. One looks at returns for orbit segments
whose length is given by a parameter $t$ scaled by the size of the target set. The scaling
factor is suggested by Kac's theorem. In~\cite{L02,KL} it was shown that for ergodic
maps one can achieve any limiting statistics if one chooses the shrinking target sets
suitably.
For a generating partition the natural neighbourhoods are cylinder sets and for those
the limiting distributions for entry and return times were shown under various mixing conditions
 to be exponential with parameter $1$
(see for instance~\cite{Hirata1,GS97,C01, Ab2}).
However Kupsa constructed an example  which has a limiting hitting time distribution 
almost everywhere and which is not the exponential distribution with parameter $1$. 
For multiple returns it has been established under various mixing conditions
that the limiting distribution is Poissonian almost surely. The first such result is due
to Doeblin~\cite{Doe} for the Gauss map for which he showed that at the origin multiple
returns to cylinder sets are in the limit Poisson distributed. Pitskel~\cite{Pit} (see also~\cite{Den})
 used the
moment method to prove that the return times are in the limit Poissonian for equilibrium
states for Axiom A maps using Markov partitions. Similar results have been shown
for some non-uniformly expanding maps~\cite{HSV,Dol}, for rational maps~\cite{H00}
and for $\phi$-mixing maps~\cite{AV1,Ab3}. Using the Chen-Stein method, it was
also shown for toral automorphism (by way of harmonic analysis)~\cite{DGS}
 and for unions of cylinders and $\phi$-mixing measures in~\cite{HP10}. 
More results arementioned in the review~\cite{H13}.

For the return times to metric balls $B_\rho$ on manifolds~\cite{CC13} proves the limiting distribution
to be Poissonian for the SRB measure on a one-dimensional attractor which allows the
construction of a Young tower~\cite{LSY98,LSY99} with exponentially decaying
correlations. The speed of convergence turns out to be a (positive) power of the radius $\rho$
of the target ball.
Here we prove a similar result in the case when the correlations decay
at a polynomial rate. Also the attractor is not required to be one-dimensional.
In fact, the result here applies to any invariant measure that can be constructed
using the tower construction whether it be absolutely continuous or not.
The speed of convergence in this case is a negative power of $\abs{\log\rho}$.
Let us note that recently P\`ene and Saussol~\cite{PS} have obtain the limiting
Poisson distribution for return times for the SRB measure assuming some
geometric regularity.

 The results of this paper are taken from~\cite{Was} and are organised as follows:
 In Section~\ref{results} we state the results where the exact conditions are given
in later sections (indicated within the theorems). In Section~\ref{YoungTower}
we give the description of Young's tower construction which is central to the result.
In Section~\ref{ProofTheorem} we prove the main result Theorem~1
which proceeds in several steps which are outlined at the beginning of the section.
In Section~\ref{VeryShortReturns} we consider the case of a diffeomorphism on
 a manifold and show that very short returns i.e.\ those that are of order
 $\abs{\log\rho}$ (where $\rho$ is the radius of the target ball) constitute a
 very small portion of the manifold whose measure can be bounded.
 We use and adapt an argument from~\cite{CC13} Lemma~4.1.

\section{Results}\label{results}
Let $(M, T)$ be a dynamical system on the compact metric space $M$. For a positive parameter
$\mathfrak{a}$ define the set
\begin{equation} \label{defM_rho,J}
\mathcal{V}_\rho(\mathfrak{a}) = \{\mathsf{x} \in M: \ball(\mathsf{x}) \cap T^{n}\ball(\mathsf{x}) \ne \varnothing \text{ for some } 1 \leq n < \mathfrak{a}\abs{\log\rho}\},
\end{equation}
where $\rho>0$. The set $\mathcal{V}_\rho$ represents the points within $M$ with very short return times.

For a ball $\ball(\mathsf{x})\subset M$ we define the counting function
\vspace{-0.2cm}
\begin{equation*}
S^{t}_{\rho, \mathsf{x}}(x)=\sum_{n=0}^{\floor{t/\mu(\ball(\mathsf{x}))}-1} {\ind_{\ball(\mathsf{x})}} \circ T^n(x).
\end{equation*}
which tracks the number of visits a trajectory of the point $x \in M$ makes to the ball $\ball(\mathsf{x})$
on an orbit segment of length $N=\floor{t/\mu(\ball(\mathsf{x}))}$.
(We often omit the sub- and superscripts and simply use $S(x)$.)

Let us now state our main results. Definitions and background material  are in Section~\ref{YoungTower}.

\begin{MT1}\label{MT1}
Let $(M, T)$ be a dynamical system which can be modeled by a Young tower. Suppose that the tail of the tower's return time function decays polynomially with degree $\lambda > 4$. Let $\mu$ be the SRB measure admitted by the system and let $\varsigma$ be its dimension. Assume that $\mu$ is geometrically and $\xi$-regular. Exact assumptions are in Section \ref{ass_T1}.

Then there exist constants $\kappa, \kappa' \in (0, \frac{\lambda-4}{2})$ and $\mathfrak{a}, C, C'>0$ such that for $\rho$ sufficiently small there exists a set $\mathcal{X}_\rho \subset M$ with $\mu(\mathcal{X}_\rho) \leq C' \abs{\log \rho}^ {- \kappa'}$ such that for all $\rho$-balls with centers
$\mathsf{x} \notin \mathcal{X}_\rho \cup \mathcal{V}_\rho(\mathfrak{a})$ we have
\begin{equation}\label{poissonian}
\Abs{\mathbb{P}(S=k) - e^{-t} \frac{t^k}{k!}} \; \leq \; C \, \abs{\log \rho}^ {- \kappa} \qquad \text{ for all } k \in \nats.
\end{equation}
\end{MT1}

\noindent Theorem~1 establishes the limiting statistics of returns outside the set $\mathcal{V}_\rho$
of very short returns. If $M$ is a manifold and the map $T$ sufficiently regular, one can say
something about the size of $\mathcal{V}_\rho$. This is done in the following result
where we obtain for smooth maps on manifolds the following
limiting result that at the same time controls the size of the forbidden set.

\begin{MT2}
Let $(M, T)$ be a dynamical system satisfying the assumptions of Theorem~1 and where $T$
is a $C^2$-diffeomorphism and $\lambda>9$. Let $\mu$ be the invariant measure. For
$\mathsf{x} \not\in \mathcal{X}_\rho \cup \mathcal{V}_\rho(\mathfrak{a})$ the function $S^{t}_{\rho, \mathsf{x}}$ counting the number of visits to the ball $\ball(\mathsf{x})$ satisfies~\eqref{poissonian} for $\rho$ sufficiently small where the constants $\kappa \in (0, \frac{\lambda-4}{2})$ and $C>0$ are independent of $\rho$ and $\mathsf{x}$.

Further there exists  $\kappa' \in (0, \frac{\lambda-9}{4}]$ and $C''>0$ so that
\begin{equation*}
\mu(\mathcal{X}_\rho \cup \mathcal{V}_\rho(\mathfrak{a})) \leq C'' \, \abs{\log \rho}^ {- \kappa'}\\[0.2cm]
\end{equation*}
for all $\rho$ small enough
($\mathfrak{a} = [\, 4 \, (\norm{DT}_{\mathscr{L}^\infty} + \norm{DT^{-1}}_{\mathscr{L}^\infty})]^{-1}$).
\end{MT2}

\noindent Note that as $\rho$ tends to zero so does the error of approximation and the measure of the excluded set. Therefore in the limit the function $S$ is Poisson distributed on a full measure set.

If the assumption ``$\mu$ geometrically and $\xi$-regular" is replaced with ``$\mu$ absolutely
continuous with respect to Lebesgue measure" (the latter condition implies the former) then
we obtain the following result which generalises the main theorem of~\cite{CC13} from
exponentially decaying correlation to polynomially decaying correlations. Moreover, here
we don't require the attractor to have one-dimensional unstable manifolds.

\begin{MT3}
Let $(M, T)$ be a dynamical system on a compact manifold $M$ where $T: M \maps M$ is a $C^2$ diffeomorphism with attractor $\mathscr{A}$. Suppose the system can be modeled by a Young tower whose return time function decays polynomially with degree $\lambda > 9$. Let $\mu$ be the SRB measure admitted by the system and let $\varsigma$ be its dimension. Exact assumptions are in
Section~\ref{SRBmeasure}.

For every  $\kappa < \frac{\lambda-4}{2}$ there exists $\hat{C}>0$, and sets
$\mathcal{Z}_\rho \subset M$ with
$$
\mu(\mathcal{Z}_\rho) \leq \hat{C} \abs{\log \rho}^{- \frac{\lambda-9}4}
$$
 so that for
 $\mathsf{x} \notin \mathcal{Z}_\rho$, the function $S^{t}_{\rho, \mathsf{x}}$ counting the number of visits to the ball $\ball(\mathsf{x})$ satisfies
\begin{equation*}
\Abs{\mathbb{P}(S=k) - e^{-t} \frac{t^k}{k!}} \; \leq \; \hat{C} \, \abs{\log \rho}^ {-\kappa} \qquad \text{ for all } k \in \nats.
\end{equation*}
\end{MT3}

\vspace{3mm}

\noindent Throughout the paper $C_1, C_0, \hdots$ and $\alpha, \beta, \hdots$ denote global constants
 while $c_0, c_1, \hdots$ are locally defined constants.

%%%%%%%%%%%%%%%%%%%%%%%%%%%%%%%%%%%%%%%%%%%%%%%%%%%%%%%%%%%%%%%%%%%%%%%
%%%%%%%%%%%%%%%%%%%%%%%%%%%%%%%%%%%%%%%%%%%%%%%%%%%%%%%%%%%%%%%%%%%%%%% Background

\section{Young Tower} \label{YoungTower}

We shall use the tower method which was developed by L-S Young in~\cite{LSY98, LSY99} to
construct invariant measures and to obtain decay rates for correlations.

\vspace{3mm}
%%%%%%%%%%%%%%%%%%%%%%%%%%%%%%%%%%%%%%%%%%
%%%%%%%%%%%%%%%   PARTITION, FOLIATION, RETURN TIME

\noindent {\bf (I) Partition, foliation and return time function:}
The base $\Lambda \subset M$ is assumed to have a suitable `reference measure' $\hat{m}$. There exists a collection of pairwise disjoint subsets $\{\Lambda_i\}_{i \in N}$ so that the following are satisfied:
\begin{itemize}
\item[a)] Modulo sets of $\hat{m}$-measure zero $\Lambda = \bigsqcup_i \Lambda_i$.
\item[b)] The base $\Lambda$ is equipped with the product structure and each of the $\Lambda_i$ is assumed to be a rectangle (see below) and satisfy the Markov property under the return transform (see~\cite{LSY98}).
\item[c)] There exists a return time function $R: \Lambda \maps \mathbb{N}$ such that $T^{R}\Lambda_i = \Lambda$.
\end{itemize}

Now, let $\gamma^u(x)$ be the local unstable leaf through $x$ and $\gamma^s(y)$ be the local stable leaf through $y$. $\Lambda_i$ being a rectangle means that if $x,y\in\Lambda_i$ then there is a unique intersection $z=\gamma^u(x)\cap\gamma^s(y)$ which also lies in $\Lambda_i$. 
The map $T$ contracts along the stable leaves and similarly $T^{-1}$ contracts along the unstable 
leaves (see Assumption~(A1)).
Note that, if $\gamma^u, \hat\gamma^u$ are two unstable leaves then the homotopy map $\Theta:\gamma^u\cap\Lambda\to \hat\gamma^u\cap\Lambda$ is defined by $\Theta(x)=\hat\gamma^u\cap\gamma^s(x)$ for $x\in\gamma^u\cap\Lambda$.

The return time function $R$ is constant on each of the subsets $\Lambda_i$, meaning that for each $i$ there is an $R_i \in \mathbb{Z}^{+}$ such that $R|_{\Lambda_i} = R_i$ and in fact $T^{R_i}\Lambda_i = \Lambda$. Without loss of generality we will assume that the greatest common divisor of all of the $R_i$ is equal to one. The function $R$ is assumed to be integrable on each unstable leaf $\gamma^u$ with respect to a `reference measure' $\hat{m}$. That is
$$
\sum_i R_i \hat{m}_{\gamma^u}(\Lambda_i) = \int_{\Lambda} R \,d\hat{m}_{\gamma^u} < \infty
$$
for all local unstable leaves $\gamma^u$,
where $\hat{m}_{\gamma^u}$ is the conditional measure on $\gamma^u$.
Precise assumptions are formulated in Assumptions~(A1) and~(A2).

Since any point $x$ from $\Lambda_i \subset \Lambda$ will return to $\Lambda$ after $R_i$ iterations we can define the return transform $\hat{T}: \Lambda \maps \Lambda$ piecewise as follows:
$$
\hat{T}x := T^{R_i}x \quad \text{for } x \in \Lambda_i.
$$
By assumption $\hat{T} \Lambda_i = \Lambda$ with $\hat{T}|_{\Lambda_i}$ one to one and onto.

We extend the return time function $\hat{T}$ to all of $M$ in the following way.
For a point $x$ let $j\ge0$ be the smallest integer so that $T^jx\in\Lambda$. Then 
we put $\hat{T}(x)=T^j(x)$.

\vspace{3mm}
%%%%%%%%%%%%%%%%%%%%%%%%%%%%%%%%%%%%%%%%%%
%%%%%%%%%%%%%%%   SEPARATION TIME
\noindent {\bf (II) Separation Time:}
Based on the flight time $R$ and the partition of the base, we define the {\em separation time}
\begin{equation*}
s(x,y) = \min \, \{k \geq 0 \, : \; \hat{T}^k x \text{ and } \hat{T}^k y \text{ lie in distinct } \Lambda_i\},
\end{equation*}
so that for $x,y \in \Lambda_i$, $s(x,y) = 1+ s(\hat{T} x, \hat{T} y)$ and in particular $s(x,y) \geq 1$.
If  $F\subset\Lambda$ then we also put
$$
s(F) =\min_{x,y\in F}s(x,y).
$$
We extend the separation function to points outside $\Lambda$ as follows. If for points $x,y$ there
exists an integer $k\ge0$ (smallest) so that $T^kx,T^ky\in\Lambda_i$ for some $i$, then
$s(x,y)=s(T^kx,T^ky)$. 

\vspace{3mm}
%%%%%%%%%%%%%%%%%%%%%%%%%%%%%%%%%%%%%%%%%%
%%%%%%%%%%%%%%%   JACOBIAN

\noindent {\bf (III) The Jacobian:}
Even though the original system $(M, \mathscr{B}, T, \mu)$ is not necessarily differentiable in the
ordinary sense, one uses the
Radon-Nikodym derivative of $T$ with respect to the `reference measure' $\hat{m}$
(following~\cite{LSY98}, p.~596). The derivative exists and is well defined because every $T^{R_i}|_{\Lambda_i}$ and its inverse are non-singular with respect to the conditional measure $\hat{m}_{\gamma^u}$ on
unstable leaves $\gamma^u$. Let
$$
JT = \frac{d(T^{-1}_{*}\hat{m}_{\gamma^u})}{d\hat{m}_{\gamma^u}}.
$$
The requirements on $JT$ will be spelled out in Assumption~(A1).

\vspace{3mm}

%%%%%%%%%%%%%%%%%%%%%%%%%%%%%%%%%%%%%%%%%%
%%%%%%%%%%%%%%%   SRB MEASURE
\noindent {\bf (IV) The SRB measure:}
According to~\cite{LSY98,LSY99} $(M, T)$ has a generalised SRB measure $\mu$ given by
$$
\mu(S) = \sum_{i=1}^{\infty} \sum_{j=0}^{R_i-1} m(T^{-j}S \cap \Lambda_i) \quad \text{for sets } S \subset M.
$$
where $m$ is the generalised SRB measures for the uniformly expanding system system
$(\Lambda, \hat{T})$.

If we denote by $m_{\gamma^u}$ the conditional measure on unstable leaves $\gamma^u$,
then $dm=dm_{\gamma^u}d\nu(\gamma^u)$ where $d\nu$ is the transversal measure.
We will refer to the portions of the tower above each $\Lambda_i$ as {\em beams} and assume that for $n \in \mathbb{N}$, there are only finitely many $i$'s for which $R_i = n$, i.e.\ for every $n$ there are only finitely many beams with that height.

The map $T$ maps each level bijectively onto the next and the last level is bijectively mapped onto all of $\Lambda$. (Note that $R_i$ may not be the first time $\Lambda_i$ returns to $\Lambda$.) In this way $m$ can be extended
to the entire tower by
$$
m(T^j F) = m(F) \quad \text{for} \quad 0 \leq j \leq R_i-1
$$
 for any $F \subset \Lambda_i$

%%%%%%%%%%%%%%%%%%%%%%%%%%%%%%%%%%%%%%%%%%%%%%%%%%%%%%%%%%%%%%%%%%%%%%%
%%%%%%%%%%%%%%%%%%%%%%%%%%%%%%%%%%%%%%%%%%%%%%%%%%%%%%%%%%%%%%%%%%%%%%% Theorem 1: generic points

\section{Proof of Theorem~1}\label{ProofTheorem}

In this section we prove Theorem 1. We begin by stating the precise assumptions necessary for the result and derive some of the consequences that follow with minimal work. In Section~\ref{approx_ball} we introduce
cylinder sets. In order to approximate the metric balls we will restrict to those cylinder set that
have that have only short returns. We then provide several results on the behaviour of cylinder sets
 under suitable applications of the return map $\hat{T}$. The succeeding Section~\ref{removed_sets}
 contains estimates concerning the portions of the space $M$ which have to be omitted in order to obtain
 good asymptotic behaviour for the long returns. This is the `forbidden set'.
Section~\ref{set_up_T1} utilizes the Poisson approximation theorem from Section~\ref{poisson} to
establish a splitting of the error term to the Poisson distribution into two parts $\mathcal{R}_1$ and
$\mathcal{R}_2$. The remainder of the section is devoted to estimating these error terms one by one. In Section~\ref{est_R1_section} we estimate the error $\mathcal{R}_1$ which comes from long term interactions and uses decay of correlations. Sections~\ref{R2.fixed.n} and~\ref{est_R2} are devoted to bounding the term $\mathcal{R}_2$
which comes from short time (but not very short time) interactions. This is the place where
the Young tower construction comes to play and where we have to use
approximations by cylinder sets in order to make careful distinctions between short returns and
long returns to balance out different contributions to the error term. In Section~\ref{optimization} the different error terms are brought together and the various parameters are optimised.

%%%%%%%%%%%%%%%%%%%%%%%%%%%%%%%%%%%%%%%%%%%%%%%
%%%%%%%%%%%%%%%%%%%%%%%%%%%%%%%%%%%%%%%%%%%%%%%
%%%%%%%%%%%%              ASSUMPTIONS
\subsection{Assumptions} \label{ass_T1}

Let $(M, T)$ be a dynamical system equipped with a metric $d$ and let $\mu$ be the SRB measure associated to the system (whose existence follows by~\cite{LSY98,LSY99} from Assumptions~(A1) and~(A2)). We will require the following:

\vspace{0.5cm}
%%%%%%%%%%%%%%%%%%%%%%%%%%%% A1
\noindent (A1) {\em Regularity of the Jacobian and the metric on the leaves} \\ There exists a constant $C_0>0$ and $\alpha \in (0,1)$ such that for any $x, y$  in $\Lambda$ with $s(x,y) \geq 1$
\begin{eqnarray*}
(a) \qquad &&\Abs{\log \, \frac{J\hat{T}x}{J\hat{T}y}} \, \leq C_0 \, \alpha^{s(\hat{T} x,\hat{T} y)}
\quad\mbox{if } \gamma^u(x)=\gamma^u(y); \\[0.3cm]
(b) \qquad &&d(\hat{T}^k x, \hat{T}^k y) \leq C_0 \, \alpha^{s(x,y)-k} \quad \text{for } 0 \leq k < s(x,y)
\quad\mbox{ if } \gamma^u(x)=\gamma^u(y) ;\\[0.3cm]
(c)\qquad && \log\prod_{k=n}^\infty\frac{J\hat{T}(\hat{T}^kx)}{J\hat{T}(\hat{T}^ky)}\le C_0\alpha^n
\quad\mbox{if } \gamma^s(x)=\gamma^s(y); \\
(d)\qquad && \frac{d\Theta^{-1}\hat{m}_{\Theta\gamma^s}}{d\hat{m}_{\gamma^s}}(x)
=\log\prod_{k=0}^\infty\frac{J\hat{T}(\hat{T}^kx)}{J\hat{T}(\hat{T}^k\Theta x)};\\[0.1cm]
(e)\qquad &&d(\hat{T}^nx,\hat{T}^ny)\le C_0\alpha^n\quad \mbox{for } n\in\mathbb{N}\quad
\mbox{if } \gamma^s(x)=\gamma^s(y).
\end{eqnarray*}

\vspace{0.3cm}
%%%%%%%%%%%%%%%%%%%%%%%%%%          A2
\noindent (A2) {\em Polynomial Decay of the Tail} \\
There exist constants $C_1$ and $\lambda>4$ such that
\begin{equation} \label{tailDecay}
\hat{m}_{\gamma^u}(R>k) \leq C_1 \, k^{-\lambda}
\end{equation}
for every unstable leaf $\gamma^u$. With regard to Assumption~(A1)(d) this condition is satisfied
for all unstable leaves if it can be verified for a one $\gamma^u$.

\vspace{0.3cm}
%%%%%%%%%%%%%%%%%%%%%%%           A3
\noindent (A3) {\em Additional assumption on $\lambda$} \\
Let $\varsigma$ be the dimension of the measure $m_{\gamma^u}$ and 
$\hat\varsigma$ the dimension of $\mu$. We will require that
\begin{equation} \label{A2}
\xi={\varsigma}(\lambda - 1)-\hat\varsigma > 1.
\end{equation}

\vspace{0.3cm}
%%%%%%%%%%%%%%%%%%%%%%%%%%%%%%%%%%%%%%%%%%%%%%% A4
\noindent (A4) {\em Regularity of the invariant measure}\\
Let $\xi={\varsigma}(\lambda - 1)-\hat\varsigma>1$ by (A3) and suppose that the positive constant $\varsigma' < \varsigma$ is fixed. There exist a set $\mathcal{E}_{\rho}\subset M$ satisfying
$\mu(\mathcal{E}_{\rho})\le\abs{\log\rho}^{-\frac{\lambda-4}3}$ so that for $\rho$ small enough:
\begin{enumerate}
\item[(a)] ($\xi$-regularity) There exists $w_0 \in (1, \xi)$ and $a>0$ so that
$$
\frac{\mu(B_{\rho+\rho^w}(\mathsf{x}) \setminus B_{\rho-\rho^w}(\mathsf{x}))}{\mu(\ball(\mathsf{x}))}
\leq \frac{1}{g(w)\abs{\log \rho}^a}
$$
for all $\mathsf{x} \not\in\mathcal{E}_{\rho}$ and $w>w_0$ where the function $g(w)$
is so that $\sum_n^\infty g(n^\beta)^{-1}<\infty$ for some $\beta<1-\frac3{\lambda-1}$.
 We say  $\mu$ is {\em $\xi$-regular}.
\item[(b)] (geometric regularity) There exists a $\varsigma'<\varsigma$, 
$\hat\varsigma''>\hat\varsigma$ satisfying~\eqref{A2} and $C_2 > 0$ such that 
$$
m_{\gamma^u}(\ball) \leq C_2 \, \rho^{\varsigma'},
\qquad  \mu(\ball)\ge C_2\rho^{\hat\varsigma''}
$$
for all $\ball(\mathsf{x}) \subset M$ for which  $\mathsf{x} \not\in \mathcal{E}_{\rho}$
and all unstable leaves $\gamma^u$.
\end{enumerate}

%%%%%%%%%%%%%%%%%%%%%%%%%%%%%%%%%%%%%%%%%%%%%
%%%%%%%%%%%%%%%%%%%%%%%%%%%%%%%%%%%%%%%%%%%%%
%%%%%%%%%%  SUBSECTION: IMMEDIATE CONSEQUENCES.....
\subsection{Immediate Consequences of the Assumptions}

In this section we list some basic results which will be needed in the proof of the main results.

%%%%%%%%%%%%%%%%%%%%%%%%%%%%%%%%%%%%%%%%%%%
%%%%%%%      LEMMA: DISTORTION
\begin{lem}[Distortion]\label{Distortion}
There exists a constant $C_3>1$ such that\\
(i) for any $x$ and $y$ in $\Lambda$, $\gamma^u(x)=\gamma^u(y)$  with separation time 
$s(x,y) \geq q$:
\begin{equation} \label{distortionUse}
J\hat{T}^q x \in J\hat{T}^q y \; \biggl[\frac{1}{C_3}, C_3 \biggr].
\end{equation}
(ii) For any $F \subset F' \subset \Lambda_i\cap\gamma^y$ (for some $i$) and for any $q \leq s(F')=\inf_{x,x'\in F'}s(x,x')$
\begin{equation} \label{pushforwards}
\frac{1}{C_3} \, \frac{m_{\hat\gamma^u}(\hat{T}^q F)}{m_{\hat\gamma^u}(\hat{T}^q F')} 
\leq \frac{m_{\gamma^u}(F)}{m_{\gamma^u}(F')}
\leq C_3\, \frac{m_{\hat\gamma^u}(\hat{T}^q F)}{m_{\hat\gamma^u}(\hat{T}^q F')},
\end{equation}
where $\hat\gamma^u=\gamma^u(\hat{T}^q(F))$.
\end{lem}

\begin{proof}
(i) Let $x,y \in \Lambda$, , $\gamma^u(x)=\gamma^u(y)$  and let $q \geq 1$ be an integer
 less than or equal to $s(x,y)$. 
Then by the chain rule and (A3)(a)
$$
\Abs{\log \, \frac{J\hat{T}^q x}{J\hat{T}^q y}} \leq \sum_{j=0}^{q-1} \Abs{\log \, \frac{J\hat{T}(\hat{T}^j x)}{J\hat{T}(\hat{T}^j y)}}
\leq \sum_{j=0}^{q-1} C_0 \alpha^{s(\hat{T}(\hat{T}^j x),\hat{T}(\hat{T}^j x))} = C_0 \sum_{j=0}^{q-1} \alpha^{q-(j+1)} \leq \frac{C_0}{1-\alpha}.
$$
as $\alpha^{s(\hat{T}^{j+1} x,\hat{T}^{j+1} y)} \leq \alpha^{q-(j+1)}$
which implies the statement~(i) with $C_3= e^{\frac{C_0}{1-\alpha}}$.

\vspace{3mm}

\noindent (ii) By the Mean Value Theorem
$$
\frac{m_{\gamma^u}(F)}{m_{\gamma^u}(F')} 
= \frac{J(\hat{T}^q x)^{-1} m_{\hat\gamma^u}(\hat{T}^q F)}
{J(\hat{T}^q x')^{-1} m_{\hat\gamma^u}(\hat{T}^q F')}
= \frac{J(\hat{T}^q x')}{J(\hat{T}^q x)} \; 
\frac{m_{\hat\gamma^u}(\hat{T}^q F)}{m_{\hat\gamma^u}(\hat{T}^q F')}
$$
for some  $x\in F$ and $x'\in F'$.
The result follows now from part~(i).
\end{proof}

%%%%%%%%%%%%%%%%%%%%%%%%%%%%%%%%%%%%%%%%%%%
%%%%%%%      DECAY OF CORRELATION

According to~\cite{LSY99}(Theorem~3) the decay of correlations is polynomial:
Let $\phi$ be a Lipschitz continuous function and $\psi\in\mathscr{L}^\infty$  constant on
 local stable leaves. Then one has
\begin{equation} \label{corrDecay}
\Abs{\int_M \phi \; \psi \circ T^n \, d\mu - \int_M \phi \, d\mu \int_M \psi \, d\mu} \; \; \leq \; \varphi_n \norm{\phi}_{Lip} \norm{\psi}_{\mathscr{L}^\infty}
\end{equation}
where the decay function $\varphi_n=\mathcal{O}(1) \sum_{k>n} m(R>k)\leq C_4 \, n^{-\lambda+1}$, for
some $C_4>0$ where $\lambda>0$  is the tail decay exponent from assumption~(A2).
Note that in general for functions $\psi$ which are not constant on local stable leaves, the supremum norm on the RHS of~\eqref{corrDecay} has to be replaced by the Lipschitz norm (see~\cite{LSY98}).\\

\vspace{3mm}

\noindent We will also need the following function of $s \in \mathbb{R}^+$:
\begin{equation*}
\Omega(s) := \sqrt{\sum_{i: R_i> s} R_i \; m(\Lambda_i)}.
\end{equation*}
Since the return time $R$ is integrable $\Omega(s) \maps 0$ as $s \maps \infty$.

%%%%%%%%%%%%%%%%%%%%%%%%%%%%%%%%%%%%%%%%%%%
%%%%%%%      LEMMA: DECAY OF OMEGA
\begin{lem}[Decay of $\Omega$] \label{omegaEst}
There exists a constant $C_5$ such that for $s \geq 4$
\begin{equation*}
\Omega(s) \leq C_5 s^{- \theta}
\end{equation*}
where $\theta = (\lambda -1)/2$.
\end{lem}

\begin{proof}
By definition
$$
\Omega(s)^2 =\sum_{i: R_i > s} R_i \, m(\Lambda_i)
 \leq \sum_{k=s}^{\infty} m(R>k) + s \, m(R>s)
 \leq \sum_{k=s}^{\infty} C_1 k^{-\lambda} + s C_1 s^{-\lambda}
 \leq  c_1 s^{2\theta}
$$
using the tail decay, where $c_1<\infty$. We complete the proof by setting $C_5 = \sqrt{c_1}$.
\end{proof}

%%%%%%%%%%%%%%%%%%%%%%%%%%%%%%%%%%%%%%%%%%%%%%
%%%%%%%%%%%%%%%%%  SUBSECTION CYLINDER SETS
\subsection{Cylinder sets}\label{approx_ball}

 Let $s$ be a given integer. We shall separate the beams with return times greater than $s$ and also
 portions of the base that visit those beams during the ``flight''.
 The beams with heights less than $s$ will be referred to as ``short'' and constitute the principal part.
 The ``tall'' beams (i.e.\ when the returns are $>s$) will be treated like error terms and contribute
 to the ``forbidden'' set $\mathcal{X}_\rho$ whose size is small and estimated in Section~\ref{removed_sets}.
Let us introduce several quantities that will be needed to deal with the long return times.\\
%%%%%%%%%%%%%%%%%%
{\bf (I)} For indices  $(i_0,\dots,i_l)\in \mathbb{N}^{l+1}$ we define the {\em $l$-cylinder}
(w.r.t.\  the map $\hat{T}$) by
$$
\zeta_{i_0, \hdots, i_l}
= \Lambda_{i_0} \cap \hat{T}^{-1}\Lambda_{i_1} \cap \hat{T}^{-2}\Lambda_{i_2} \cap \hdots \cap \hat{T}^{-l}\Lambda_{i_l},
$$
and denote by $\mathfrak{I}$ the collection of indices $(i_0,\dots,i_l)$
such that the associated cylinder $\zeta_{i_0, \hdots, i_l}$ is non-empty.\\
%%%%%%%%%%
{\bf (II)} For every $\Lambda_i$ let's define the subset consisting exclusively of points that only visit short beams:
$$
\tilde{\Lambda}_i = \{x \in \Lambda_i: \forall l \leq n, \; R(\hat{T}^{l} x) \leq s \}.
$$
Let us note that if the original beam $\Lambda_i$ happens to be tall (i.e. $R_i>s$)  then the corresponding
$\tilde{\Lambda}_i$ will be empty. With $\Lambda=\bigcup_i\Lambda_i$ and
$\tilde\Lambda=\bigcup_i\tilde\Lambda_i$ we obtain in particular
$\Lambda\setminus\tilde\Lambda=\{x\in\Lambda: \exists \, l\in[0,n] \mbox{ s.t. } R(\hat{T}^{l} x) \leq s\}$.
Similarly we define the restriction of a cylinder $\zeta_{i_0, \hdots, i_l}$ to short returns  by
$$
\tilde{\zeta}_{i_0, \hdots, i_l} =\left\{x\in \zeta_{i_0, \hdots, i_l}: R(\hat{T}^j(x))\le s\;\forall\; j=0,\dots,l\right\}.
$$
In other words
\begin{equation} \label{cylinderTildeMeasure}
\tilde{\zeta}_{i_0, \hdots, i_l} =
\begin{cases}
{\zeta}_{i_0, \hdots, i_l}& \text{if } R_{i_0},\dots,R_{i_l}\le s \\
\varnothing & \text{otherwise}.
\end{cases}
\end{equation}
In particular we see that if one of the beams on the cylinder's path is tall,
i.e. $R_{i_j}>s$ for a $j\in[0,l]$,  then it must have originated inside
$\Lambda_{i_0}\setminus\tilde\Lambda_{i_0}$. Thus
\begin{equation}\label{long.cylinders}
\bigcup_{i_0, \hdots, i_l}{\zeta}_{i_0, \hdots, i_l}\setminus\tilde{\zeta}_{i_0, \hdots, i_l}
\subset\bigcup_{i_0}\Lambda_{i_0}\setminus\tilde\Lambda_{i_0}
\hspace{1cm}\mbox{ and } \hspace{1cm}
\bigcup_i\tilde\Lambda_i
\subset\bigcup_{i_0, \hdots, i_l}\tilde{\zeta}_{i_0, \hdots, i_l}
\end{equation}
as long as  $l\le n$.\\
%%%%%%%%%%%%%%
{\bf (III)} For given $n$, $j$  and $i_0 \in \mathbb{N}$, $j < R_{i_0} \leq s$ we define the set of
 `suitable' symbols by
\begin{equation} \label{index set}
I_{i_0, j, n} = \biggl \{ (i_0, \hdots, i_l)\in \mathfrak{I}:
\sum_{k=0}^{l-1} R_{i_k}\le n+j <\sum_{k=0}^l R_{i_k}\biggr \}.
\end{equation}
Accordingly a string of symbols $(i_0, \hdots, i_l)$ is called {\em $(n,j)$-minimal} if it
satisfies the property $R^l \leq n+j < R^{l+1}$, where $R^l=\sum_{k=0}^{l-1} R_{i_k}$.
Note that for any given values $j$ and $n$ the cylinders indexed by $I_{i, j, n}$ partition the beam base
$\Lambda_{i}$, up to set of measure zero, i.e.\
$$
{\Lambda}_i = \bigsqcup_{\tau \in I_{i, j, n}} \hspace{-0.25cm}{\zeta}_{\tau}.
$$
In what follows we shall often write $I$, instead of $I_{i, j, n}$.

\vspace{3mm}

\noindent If $\tau = (i_0, \hdots, i_l)$ then let $\tau'=(i_0, \hdots, i_{l-1})$. Since
${\zeta}_{\tau} = {\zeta}_{\tau'} \cap \hat{T}^{-l}\Lambda_{i_l} \subset {\zeta}_{\tau'}$, we will write
 $\tau \subset \tau'$ to reflect the relationship between the cylinders.

\vspace{3mm}

\noindent The remainder of this section is taken up by providing some essential estimates
involving the quantities introduced.

%%%%%%%%%%%%%%%%%%%%%%%%%%%%%%%%%%%%%%%%%%%%
%%%%%%%%%%%%%%%%%%%%%%%%%%% Lemma
\begin{lem} \label{cylinderDiameterLem}
The diameter of the cylinder set $\zeta_{i_0, \hdots, i_l}$ restricted to unstable leaves is exponentially small:
\begin{equation}
|\zeta_{i_0, \hdots, i_l}\cap \gamma^u| \leq C_0 \, \alpha^{l+1}.
\end{equation}
\end{lem}

\begin{proof}
Let $x$ and $y$ be two points in $\zeta_{i_0, \hdots, i_l}$, then by definition we have
\begin{equation*}
\hat{T}^k x, \hat{T}^k y \in \Lambda_{i_k}\cap\gamma^u \quad \text{for} \quad 0 \leq k \leq l.
\end{equation*}
It follows that $s(x,y) \geq l+1$. Therefore by Assumption~(A1b):
\begin{equation*}
d(x,y) \leq C_0 \, \alpha^{s(x,y)} \leq C_0 \, \alpha^{l+1}
\end{equation*}
and, since the points $x,y$ were arbitrary, we conclude that
\begin{equation*}
|\zeta_{i_0, \hdots, i_l}\cap\gamma^u| \leq C_0 \, \alpha^{l+1}.
\end{equation*}
\end{proof}

%%%%%%%%%%%%%%%%%%%%%%%%%%%%%%%%%%%%%%%%%
%%%%%%%%%%%%%%%%%%% LEMMA RATIOS
\begin{lem} \label{case1}
Let $\tau = (i_0, \hdots, i_l)$ be an index in $I_{i_0, j, n}$ and put $\tau' = (i_0, \hdots, i_{l-1})$. Then
$$
\frac{m_{\gamma^u}(T^{-j}\mathcal{B} \cap \tilde{\zeta}_{\tau'})}
{m_{\gamma^u}(\tilde{\zeta}_{\tau'})}
\leq C_6 \, \mu(\ball)
$$
for any set $\mathcal{B}\subset T^{-n}B_\rho$ and for all $\tilde\zeta_{\tau'}\not=\varnothing$.
\end{lem}

\begin{proof}
By inclusion we have
$$
\frac{m_{\gamma^u}(T^{-j}\mathcal{B} \cap \tilde{\zeta}_{\tau'})}
{m_{\gamma^u}(\tilde{\zeta}_{\tau'})}
\leq \frac{m_{\gamma^u}(T^{-(n+j)}\ball \cap \tilde{\zeta}_{\tau'})}
{m_{\gamma^u}(\tilde{\zeta}_{\tau'})}.
$$
Let the number $b$ be such that $n+j-b = R^l$. Recall that $\tau \in I$ means that it is $(n,j)$-minimal,
 i.e.\ that $n+j$ lies between $R^l$ and $R^{l+1}$, we have $0 \leq b < R_{i_l}$. Further $T^{n+j-b} = \hat{T}^l$. Recall that $s({\zeta}_{\tau'}) = l$, thus we can push both the numerator and the denominator forward by $\hat{T}^l$ and use distortion, Lemma~\ref{Distortion}
and~(A1)(a), to obtain for $\tilde\zeta_{\tau'}\not=\varnothing$:
\begin{equation}\label{comparison}
\frac{m_{\gamma^u}(T^{-(n+j)}\ball \cap \tilde{\zeta}_{\tau'})}{m_{\gamma^u}(\tilde{\zeta}_{\tau'})}
\leq C_3 \, \frac{m_{\hat\gamma^u}(T^{-b}\ball \cap \hat{T}^l{\zeta}_{\tau'})}
{m_{\hat\gamma^u}(\hat{T}^l{\zeta}_{\tau'})}
\end{equation}
since by assumption and~\eqref{cylinderTildeMeasure} $\tilde\zeta_{\tau'}=\zeta_{\tau'}$.
Here we put again $\hat\gamma^u=\gamma^u(\hat{T}^lx)$ for $x\in\zeta_{\tau'}\cap\gamma^u$.
Now $\hat{T}^l({\zeta}_{\tau'}\cap\gamma^u) = \Lambda\cap\hat\gamma^u$, 
because $s({\zeta}_{\tau'}) = l$.
 For the numerator we obtain
$$
m_{\hat\gamma^u}(T^{-b}\ball \cap \Lambda) 
\le c_1\int m_{\hat\gamma^u}(T^{-b}\ball \cap \Lambda)\,d\nu(\hat\gamma^u)
\le c_1\mu(T^{-b}\ball)
=\mu(\ball)
$$
for some $c_1$, and for the denominator we use that 
$\mu_{\hat\gamma^u}(\Lambda)\ge c_2$ for some $c_2>0$. The lemma now follows
with $C_6=C_3c_1/c_2$.
\end{proof}

%%%%%%%%%%%%%%%%%%%%%%%%%%%%%%
%%%%%%%%%%%%%%%%%%%%%%%%%%%%%%
%%%%%%%%%%% Lemma, unions of cylinders
\begin{lem} \label{sumCylind}
Consider a collection of cylinders $\zeta_{\tau'} = \zeta_{i_0, \hdots, i_{l-1}} \in \Lambda_{i_0}$ such that
\begin{align*}
& i) \quad \, \exists \tau \subset \tau' \text{ with } \tau \in I_{i_0,j,n}, \\
& ii) \quad \zeta_{\tau'} \cap T^{-j}\ball \ne \varnothing.
\end{align*}
Then there exists a constant $C_7$ such that
$$
\sum_{\substack{\tau' | \exists \tau \subset \tau' \\ \tau \in I \\ T^{-j}\ball \, \cap \, {\zeta}_{\tau'} \ne \varnothing}} \hspace{-0.6cm} m({\tilde\zeta}_{\tau'}) \leq m(\gamma^s(T^{-j} (B_{\rho + C_7 \, \alpha^{n/s}})) \cap \Lambda_{i_0}),
$$
where we use the notation $\gamma^s(\mathcal{B})
=
\bigcup_{\gamma^s: \gamma^s\cap \mathcal{B}}\gamma^s$.
\end{lem}

\begin{proof}
All the unions, sums and maxima in this proof are subscripted with ``$\tau' \, | \, \exists \tau \subset \tau', \tau \in I_{i_0,j,n}, T^{-j}\ball \, \cap \, {\zeta}_{\tau'} \ne \varnothing$," unless otherwise specified.
 From Lemma~\ref{cylinderDiameterLem} we know that
$$
|{\zeta}_{\tau'}\cap\gamma^u| \leq C_0 \, \alpha^{s({\zeta}_{\tau'})} = C_0 \, \alpha^{l}
$$
for unstable leaves $\gamma^u$.
Without loss of generality we can assume that $\tilde\zeta_{\tau'}=\zeta_{\tau'}$.
Since ${\zeta}_{\tau'}$ contains a cylinder $\tau$ satisfying $(n,j)$-minimality, we deduce from
$n+j < \sum_{k=0}^{l} R_{i_k} \leq (l+1) \, s$ a lower bound on $l$:
\begin{equation*}
l \geq \frac{n+j}{s} -1 \geq \frac{n}{s} -1.
\end{equation*}
Thus
$$
|{\zeta}_{\tau'}\cap\gamma^u| \leq C_0 \, \alpha^{\frac{n}s -1}
$$
and for $j<R_{i_0}$ we can further say by Assumption~(A1)(b) that
$$
|T^j ({\zeta}_{\tau'}\cap\gamma^u)| \leq  C_0\, \alpha^{\frac{n}s-2}
$$
as $s(T^j\zeta_{\tau'})=s(\zeta_{\tau'})-1$ for $1\le j<R_{i_0}$.
Now put $C_7=c_0\alpha^{-2}$.
Since $\zeta_{\tau'} \cap T^{-j}\ball \ne \varnothing$ and therefore $T^j \zeta_{\tau'} \cap \ball \ne \varnothing$
we obtain
$$
T^j {\zeta}_{\tau'} \subset 
\bigcup_{\gamma^s: T^j\gamma^s\cap B_{\rho + C_7 \, \alpha^{n/s}}\not=\varnothing}T^j\gamma^s.
$$
As the above estimate works for any cylinder whose subscript $\tau'$ satisfies ``$\tau' \, | \, \exists \tau \subset \tau', \tau \in I_{i_0,j,n}, T^{-j}\ball \, \cap \, {\zeta}_{\tau'} \ne \varnothing$," we have
$$
\bigcup {\zeta}_{\tau'} \subset  \gamma^s(T^{-j}(B_{\rho + C_7 \, \alpha^{n/s}})).
$$
Moreover, since $\bigcup {\zeta}_{\tau'} \subset \Lambda_{i_0}$ we in fact have
\begin{equation*}
\bigcup {\zeta}_{\tau'} \subset \gamma^s(T^{-j} ((B_{\rho + C_7 \, \alpha^{n/s}})) \cap \Lambda_{i_0},
\end{equation*}
and we can conclude
$$
\sum m({\zeta}_{\tau'}) =m \biggl(\bigcup {\zeta}_{\tau'} \biggr)
\leq m(\gamma^s(T^{-j} (B_{\rho + C_7 \, \alpha^{n/s}})) \cap \Lambda_{i_0}).
$$
\end{proof}

%%%%%%%%%%%%%%%%%%%%%%%%%%%%%%%%%%%%%%%%%%%%%
%%%%%%%%%%%%%%%%%%%%%%%%%%%%%%%%%%%%%%%%%%%%%
%%%%%%%%%%%%%% SUBSECTION ON FORBIDDEN SET
\subsection{Measure of the forbidden set $\mathcal{X}_\rho$} \label{removed_sets}

We will need the following lemma.

%%%%%%%%%%%%%%%%%%%%%%%%%%%%%%%%%%%%%%%%%%%%%%%%%%%%%%%%%%%%%%%%%%%%%%%%%%%%%%% Lemma 2

\begin{lem} \label{tallTowersBad}(\cite{CC13} Lemma~A.3)
Let $\ell_0$ and $\ell_1$ be two finite positive measures on a $D$-dimensional Riemannian manifold $M$. For $\omega \in (0,1)$ and $\rho \in (0,1)$ define
\begin{equation*}
\mathcal{D} = \{ \mathsf{x} \in M: \, \ell_1(\ball(\mathsf{x})) \geq \omega \ell_0(\ball(\mathsf{x})) \}. \end{equation*}
There exists an integer $p(D)$ such that
\begin{equation*}
\ell_0(\mathcal{D}) \leq p(D) \, \omega^{-1} \ell_1(M).
\end{equation*}
\end{lem}

\noindent We will also need the following result on the size of the set where tall
towers dominate, that is where $R_i$ is larger than $s$. Recall that
$\Omega(s) = [\, \sum_{i, R_i>s} R_i \, m(\Lambda_i) \,]^{\frac{1}{2}} $.

%%%%%%%%%%%%%%%%%%%%%%%%%%%%%%%%%%%%%%%%%%%%
%%%%%%%%%%%%%%%%%%%    TALL TOWERS LEMMA #2
\begin{lem} \label{2tallTowersEst}
For $n,s\ge1$ there exist sets $\mathcal{D}_{n,s}\subset M$ such that
the non-principal part contributions are estimated as
$$
\sum_{i} \sum_{j=0}^{R_i-1} m(T^{-j}\mathcal{B} \cap (\Lambda_i \setminus \tilde{\Lambda}_i))
 <\sqrt{n+2}\,\Omega(s)\mu(\ball)
$$
for any $\mathcal{B}\subset B_\rho(\mathsf{x})$ and  $\mathsf{x}\not\in\mathcal{D}_{n,s}$  where
($p(D)$ as above)
$$
\mu(\mathcal{D}_{n,s}) \leq p(D) \sqrt{n+2}\,\Omega(s).
$$
\end{lem}

\begin{proof}
We employ Lemma~\ref{tallTowersBad},  with $\ell_0 = \mu$ and
$\ell_1(\cdot)
=\sum_{i} \sum_{j=0}^{R_i-1} m(T^{-j}(\cdot) \cap (\Lambda_i \setminus \tilde{\Lambda}_i))$.
Define
$$
\mathcal{D}_{n,s} = \{ \mathsf{x} \in M: \, \ell_1(\ball(\mathsf{x})) \geq \sqrt{n+2}\,\Omega(s) \mu(\ball(\mathsf{x})) \};
$$
from Lemma~\ref{tallTowersBad} we know that
$\mu(\mathcal{D}_{n,s}) \leq p(D) \, ( \sqrt{n+2}\,\Omega(s))^{-1} \ell_1(M)$.
Since $R(\hat{T}^{l} x) > s$ exactly if $x \in \hat{T}^{-l} \{ R > s \}$ we get
$$
\Lambda_i \setminus \tilde{\Lambda}_i
= \{x \in \Lambda_i: \exists \, l \leq n, \text{ such that } R(\hat{T}^l x) \geq s \ \}
= \bigcup_{l=0}^n \hat{T}^{-l} \{ R \geq s \} \cap \Lambda_i.
$$
Since $\tilde\Lambda_i=\varnothing$ for $R_i>i$ we bound the  measure $\ell_1(M)$  as follows
\begin{eqnarray*}
\ell_1(M) &\le& \sum_{i, R_i \leq s} \sum_{j=0}^{R_i-1} m(\Lambda_i \setminus \tilde{\Lambda}_i)
+\sum_{i, R_i > s} \sum_{j=0}^{R_i-1} m(\Lambda_i)\\
& \leq& \sum_{i, R_i \leq s} s \, m(\Lambda_i \setminus \tilde{\Lambda}_i)
+\sum_{i, R_i > s} R_i\, m(\Lambda_i)\\
 & \leq& s \sum_{i, R_i \leq s} m \biggl( \bigcup_{l=0}^n \hat{T}^{-l} \{ R > s \} \cap \Lambda_i \biggr) +\Omega(s)^2\\
 & \leq& s\, m \biggl( \bigcup_{l=0}^n \hat{T}^{-l} \{ R > s \} \biggr) +\Omega(s)^2 \\
 &\leq& s (n+1) m( \{ R > s \}) +\Omega(s)^2\\[0.2cm]
&\le&(n+2)\Omega(s)^2
\end{eqnarray*}
as $s\,m(\{R>s\})\le\sum_{j=s+1}^\infty j\,m(\{R=j\})=\Omega(s)^2$.
Hence
$$
\mu(\mathcal{D}_{n,s}) \leq p(D)  ( \sqrt{n+2}\,\Omega(s))^{-1}  \ell_1(M) =  p(D) \sqrt{n+2}\,\Omega(s).
$$
Outside the set $\mathcal{D}_{n,s}$ we have
$$
\sum_{i, R_i \leq s} \sum_{j=0}^{R_i-1}  m(T^{-j}\mathcal{B}\cap (\Lambda_i \setminus \tilde{\Lambda}_i))
 \leq \sum_{i, R_i \leq s} \sum_{j=0}^{R_i-1} m(T^{-j}(\ball) \cap (\Lambda_i \setminus \tilde{\Lambda}_i)) = \ell_1(\ball)
\leq \sqrt{n+2}\,\Omega(s)\mu(\ball).
$$
\end{proof}

\vspace{3mm}

\noindent Let $\eta\in(\frac3{\lambda-1},\frac12)$ and $\beta<1-\frac3{\lambda-1}$ be according to Assumption~(A4).
Put  $\hat\sigma=\frac{\lambda-1}2\min\left\{\eta, (1-\beta)\right\}-\frac32$. By assumption $\hat\sigma>0$
and we can finally estimate the size of the forbidden set defined by
$$
\mathcal{X}_\rho
 = \bigcup_{n=J}^{p-1}\left(\mathcal{D}_{n,n^\eta} \cup\mathcal{D}_{n,n^{1-\beta}} \right)
\cup\mathcal{E}_\rho
$$
whose parts have essentially been estimated in the previous  lemma and assumption~(A4).

%%%%%%%%%%%%%%%%%%%%%%%%%%%%%%%%%%%%%%%%%%%%
%%%%%%%%%%%%%%%%%%%   PROPOSITION FORBIDDEN SET
\begin{prop}\label{forbidden.set}
There exist a constant $C_8$ such that
$$
\mu(\mathcal{X}_\rho) \le C_8\abs{\log\rho}^{-\hat\sigma}.
$$
\end{prop}

\begin{proof}  We estimate the contributions to $\mathcal{X}_\rho$ separately in the three
following paragraphs.\\
{\bf (I)}  By Lemma~\ref{2tallTowersEst}
and since $\Omega(s)\lesssim s^{-\frac{\lambda-1}2}$
$$
\mu \biggl(\bigcup_{n=J}^{p-1} \mathcal{D}_{n,n^\eta} \biggr) \leq p(D) \sum_{n=J}^{p-1}\sqrt{n+2}\,\Omega(n^\eta)
 \leq c_1 \sum_{n=J}^{\infty} n^{\frac{1}{2}} (n^{\eta})^{-\frac{\lambda-1}{2}}
\le c_2J^{\frac{\eta+3-\eta\lambda}2}
$$
as $\eta+1-\eta\lambda<0$ and  $J = \lfloor\mathfrak{a} \, \abs{\log \rho}\rfloor$.
The term with $s=n^{1-\beta}$ is estimated similarly.
 Therefore
$$
\mu \biggl( \bigcup_{n=J}^{p-1}\mathcal{D}_{n,n^\eta}  \biggr)
+ \mu \biggl(\bigcup_{n=J}^{p-1} \mathcal{D}_{n,n^{1-\beta}} \biggr) \leq c_3 \abs{\log \rho}^{-\hat\sigma}.
$$
{\bf (II)} By Assumption~(A4) $\mu(\mathcal{E}_\rho)\le\abs{\log\rho}^{-\lambda/2}$.

\vspace{3mm}

\noindent
Combining the estimates from (I) and (II) results in
$$
\mu(\mathcal{X}_\rho) \leq \mu \biggl(\bigcup_{n=J}^{p-1} \mathcal{D}_{n,s} \biggr)
+\mu \biggl(\bigcup_{n=J}^{p-1} \mathcal{D}_{n,n^{1-\beta}} \biggr)
 +\mu(\mathcal{E}_\rho)
\le C_8\abs{\log\rho}^{-\hat\sigma}
$$
 for some $C_8$ as $\hat\sigma<\frac\lambda2$.
\end{proof}

%%%%%%%%%%%%%%%%%%%%%%%%%%%%%%%%%%%%%%%%%%%%
%%%%%%%%%%%%%%%%%%%%%%%%%%%%%%%%%%%%%%%%%%%%
%%%%%%%%%%%%%%%%%%%%%%%%%%% Set up the proof
\subsection{Poisson approximation of the return times distribution} \label{set_up_T1}
To prove Theorem~1 we will employ the Poisson approximation theorem from Section~\ref{poisson}.
Let $\mathsf{x}$ be a point in the phase space and $\ball := \ball(\mathsf{x})$ for $\rho>0$.
 Let $X_n=\ind_{\ball} \circ T^{n-1}$, then we put $N = \floor{t/\mu(\ball)}$,
 where $t$ is a positive parameter. We write $S_a^b=\sum_{n=a}^bX_n$ (and $S=S_1^N$).
Then for any $2 \leq p \leq N$ ($C_{12}$ from Section~\ref{poisson})
\begin{equation} \label{errorUSE}
\Abs{\mathbb{P}(S=k) - \frac{t^k}{k!} \, e^{-t}} \; \leq \; C_{12} ( N(\mathcal{R}_1+\mathcal{R}_2) + p \, \mu(\ball)),
\end{equation}
where
\begin{align*}
\mathcal{R}_1 &= \sup_{\substack {0 <j<N-p \\ 0<q<N-p-j}} \left|\ev(\ind_{\ball} \ind_{S_{p+1}^{N-j}=q})
- \mu(\ball) \, \ev(\ind_{S_{p+1}^{N-j}=q})\right| \\
\mathcal{R}_2 &= \sum_{n=1}^{p-1} \ev(\ind_{\ball} \; \ind_{\ball} \circ T^n).
\end{align*}
Since we restrict to the complement of the set $\mathcal{V}_\rho$ (cf.~\eqref{defM_rho,J})
we have from now on
$$
\mathcal{R}_2 = \sum_{n=J}^{p-1} \mu(\ball \cap T^{-n} \ball),
$$
where  $J = \floor{\mathfrak{a} \, \abs{\log \rho}}$.
 Note that if $k > N$ then $\mathbb{P}(S=k)=0$ and
\begin{equation} \label{yay!_k>N}
\Abs{\mathbb{P}(S=k) - \frac{t^k}{k!} \, e^{-t}} = \frac{t^k}{k!} \, e^{-t} \leq \abs{\log \rho}^{-\frac{\lambda-4}{2}} \qquad \forall k>N
\end{equation}
using the fact that $\mu(B_\rho)\lesssim\rho^{\varsigma'}$ and for $\rho$ sufficiently small.

\vspace{0.5cm}
\noindent We now proceed to estimate the error between the distribution of $S$ and a Poissonian for
$k \leq N$ based on Theorem~\ref{helperTheorem}.
%%%%%%%%%%%%%%%%%%%%%%%%%%%%%%%%%%%%%%%%%%%
%%%%%%%%%%%%%%%%%%%%%%%%%%%%%%%%%%%%%%%%%%%
%%%%%%%%%%%%%%%%% ESTIMATING R1
\subsection{Estimating $\mathcal{R}_1$} \label{est_R1_section}

By invariance of the measure $\mu$ we can also write
$$
\mathcal{R}_1 = \sup_{\substack {0 <j<N-p \\ 0<q<N-p-j}}\left| \mu(\ball \cap T^{-p}\{S_{1}^{N-j-p}=q\}) - \mu(\ball) \, \mu(\{S_{1}^{N-j-p}=q\}) \right|.
$$
We now use the decay of correlations~\eqref{corrDecay} to obtain an estimate for
$\mathcal{R}_1$. Approximate $\ind_{\ball}$ by Lipschitz functions
from above and below as follows:
\begin{equation*}
\phi(x) =
\begin{cases}
1 & \text{on $\ball$} \\
0 & \text{outside $B_{\rho + \delta \rho}$}
\end{cases}
\hspace{0.7cm} \text{and} \hspace{0.7cm}
\tilde{\phi}(x) =
\begin{cases}
1 & \text{on $B_{\rho - \delta \rho}$} \\
0 & \text{outside $\ball$}
\end{cases}
\end{equation*}
with both functions linear within the annuli. The Lipschitz norms of both $\phi$ and $\tilde{\phi}$ are equal to $1/\delta\rho$ and $\tilde{\phi} \leq \ind_{\ball} \leq \phi$.

We obtain
\begin{align*}
\mu(\ball \cap \{S_{p}^{N-j}=q\}) - \mu(\ball) \, \mu(\{S_{1}^{N-j-p}=q\})\hspace{-3cm} \\
& \leq \int_M \phi \; (\ind_{S_p^{N-j}=q}) \, d\mu - \int_M \ind_{\ball} \, d\mu \, \int_M \ind_{S_1^{N-j-p}=q} \, d\mu \\[0.2cm]
& =X+Y
\end{align*}
where
\begin{align*}
X&=\left(\int_M \phi \, d\mu - \int_M \ind_{\ball} \, d\mu \right) \int_M \ind_{S_1^{N-j-p}=q} \, d\mu\\
Y&=\int_M \phi \; (\ind_{S_p^{N-j}=q} ) \, d\mu - \int_M \phi \, d\mu \, \int_M \ind_{S_1^{N-j-p}=q} \, d\mu .
\end{align*}
The two terms $X$ and $Y$ are estimated separately.
The first term is estimated as follows:
$$
X \leq\int_M \ind_{S_1^{N-j-p}=q} \, d\mu \, \int_M (\phi - \ind_{\ball}) \, d\mu
 \leq \mu(B_{\rho + \delta \rho} \setminus \ball).
$$
In order to estimate the second term $Y$ we use the decay of correlations and
have to approximate $\ind_{S_1^{N-j-p}=q}$ by a function which is constant on local stable leaves.
For that purpose put
$$
\mathcal{S}_n
=\bigcup_{\substack{\gamma^s\\T^n\gamma^s\subset B_\rho}}T^n\gamma^s,
\hspace{6mm}
\partial\mathcal{S}_n
=\bigcup_{\substack{\gamma^s\\T^n\gamma^s\cap B_\rho\not=\varnothing}}T^n\gamma^s
$$
and
$$
\mathscr{S}_p^{N-j}=\bigcup_{n=p}^{N-j}\mathcal{S}_n,
\hspace{6mm}
\partial\mathscr{S}_p^{n-j}=\bigcup_{n=p}^{N-j}\partial\mathcal{S}_n.
$$
The set
$$
\mathscr{S}_p^{N-j}(q)=\{S_p^{N-j}=q\}\cap\mathscr{S}_p^{N-j}
$$
is then a union of local stable leaves. This follows from the fact that by construction
$T^n y\in B_\rho$ if and only if $T^n\gamma^s(y)\subset B_\rho$.
We also have
$\{S_p^{N-j}=q\}\subset\tilde{\mathscr{S}}_p^{N-j}(q)$
where the set $\tilde{\mathscr{S}}_p^{N-j}(q)=\mathscr{S}_p^{N-j}(q)\cup\partial\mathscr{S}_p^{N-j}$
is a union of local stable leaves.

Denote by  $\psi_p^{N-j}$ the characteristic function of $\mathscr{S}_p^{N-j}(q)$
and by $\tilde\psi_p^{N-j}$ the characteristic function of
$\tilde{\mathscr{S}}_p^{N-j}(q)$. Then $\psi_p^{N-j}$ and $\tilde\psi_p^{N-j}$
are constant on local stable leaves and satisfy
$$
\psi_p^{N-j}\le\ind_{S_p^{N-j}=q}\le\tilde\psi_p^{N-j}.
$$
Since $\{y:\psi_p^{N-j}(y)\not=\tilde\psi_p^{N-j}(y)\}\subset\partial\mathscr{S}_p^{N-j}$
we need to estimate the measure of $\partial\mathscr{S}_p^{N-j}$.

For integers $n$ and $s$ let $\tilde\Lambda_i$ be as before, then by Lemma~\ref{2tallTowersEst}
for $\mathsf{x}\not\in\mathcal{D}_{n,s}$
we have
$$
\sum_{i}\sum_{j=0}^{R_i-1}m(T^{-j}\mathcal{B}\cap(\Lambda_i\setminus\tilde\Lambda_i))
\le\sqrt{n+2}\,\Omega(s)\mu(B_{\rho+\alpha^l}),
$$
where $\mathcal{B}=B_{\rho+\alpha^l}\setminus B_{\rho-\alpha^l}$.
For points $y\in\Lambda\setminus\tilde\Lambda$ we let $l$ be so that $R^l(y)\le n<R^{l+1}(y)$
then we get $l\ge n/s\ge n^{\beta}$ where we choose $s(n)=\floor{n^{1-\beta}}$ and  $\beta<1$ so
that $\sum_n^\infty g(n^\beta)^{-1}<\infty$ in accordance with Assumption~(A4). By the contraction property
  $\mbox{diam}(T^n\gamma^s(y))\le\alpha^{n/s}\le \alpha^{n^{\beta}}$ for all
  $y\in\Lambda\setminus\tilde\Lambda$. Consequently
  $$
  \bigcup_{\substack{\gamma^s\subset\Lambda\setminus\tilde\Lambda\\T^n\gamma^s\subset B_\rho}}T^n\gamma^s
  \subset B_{\rho+\alpha^{l}}\setminus B_{\rho-\alpha^{l}}
  $$
and therefore
\begin{eqnarray*}
\mu(\partial\mathscr{S}_p^{N-j})
&\le&\mu\left(\bigcup_{n=p}^{N-j}T^{-n}\left(B_{\rho+\alpha^l}\setminus B_{\rho-\alpha^l}\right)\right)\\
&\le&\sum_{n=p}^{N-j}\mu(B_{\rho+\alpha^{l}}\setminus B_{\rho-\alpha^{l}})\\
&\le& \sum_{n=p}^\infty\mu(B_\rho)\frac{1}{g(w)\abs{\log\rho}^a}\\
&\le&c_1\mu(B_\rho)\frac1{\abs{\log\rho}^a}
\end{eqnarray*}
where  we used $w(n)=n^{\beta}\frac{\log\alpha}{\log\rho}$. If we split $p=p'+p''$
then we can estimate as follows:
\begin{align*}
Y&=\left|  \int_M \phi \; T^{-p'}(\ind_{S_{p''}^{N-j-p'}=q} ) \, d\mu
- \int_M \phi \, d\mu \, \int_M \ind_{S_1^{N-j-p}=q} \, d\mu\right|
\hspace{-9cm}\\
&\le \varphi_{p'}\|\phi\|_{Lip}\|\ind_{\tilde{\mathscr{S}}_{p''}^{N-j-p'}}\|_{\mathscr{L}^\infty}
 +2\sum_{n=p''}^\infty\sum_{i}\sum_{j=0}^{R_i-1}m(T^{-j}(B_{\rho+\alpha^l}\setminus B_{\rho-\alpha^l})\cap(\Lambda_i\setminus\tilde\Lambda_i))+2\mu(\partial\mathscr{S}_{p''}^{N-j})
  \end{align*}
where the triple sum on the RHS is by Lemma~\ref{2tallTowersEst} bounded by
$$
2\sum_{n=p''}^\infty\sqrt{n+2}\,\Omega(s)\mu(B_{\rho+\alpha^l})
 \leq c_2\mu(B_\rho)\sum_{n=p''}^\infty n^{\frac12-(1-\beta)\frac{\lambda-1}2}
 \leq c_3\mu(B_\rho)p^{\frac32-(1-\beta)\frac{\lambda-1}2}
$$
assuming $\frac32-(1-\beta)\frac{\lambda-1}2<0$ where in the last estimate we put $p'=p/2$.
Now let $\delta \rho = \rho^w$ where $w \in (1, \xi)$ is chosen in accordance with Assumption~(A4)
(this is possible since $\xi = {\varsigma}(\lambda - 2)>1$). Hence
\begin{align*}
\mu(\ball \cap T^{-p} \{S_{1}^{N-j-p}=q\}) - \mu(\ball) \, \mu(\{S_{1}^{N-j-p}=q\}) \hspace{-5cm}\\
&\leq \varphi_{p/2} / \delta \rho + \mu(\ball \setminus B_{\rho - \delta \rho})+c_3\mu(B_\rho)\left(p^{\frac32-(1-\beta)\frac{\lambda-1}2}+\abs{\log\rho}^{-a}\right)\\
&\leq \varphi_{p/2}\rho^{-w}+c_4\,\mu(B_\rho)\left(p^{\frac32-(1-\beta)\frac{\lambda-1}2}+\abs{\log\rho}^{-a}\right)
\end{align*}

In the same way we obtain a lower estimate. Since $\mathcal{D}_{n,n^{1-\beta}}\subset\mathcal{X}_\rho$
for $n=J,J+1,\dots$
we conclude that for $\mathsf{x}\not\in\mathcal{X}_\rho$ one has:
\begin{equation} \label{R1est}
\mathcal{R}_1
\leq \varphi_{p/2} \rho^{-w}+c_4\,\mu(B_\rho)\left(p^{\frac32-(1-\beta)\frac{\lambda-1}2}+\abs{\log\rho}^{-a}\right).
\end{equation}

%%%%%%%%%%%%%%%%%%%%%%%%%%%%%%%%%%%%%%%%%%%
%%%%%%%%%%%%%%%%%%%%%%%%%%%%%%%%%%%%%%%%%%%
%%%%%%%%%%%%%%%%% ESTIMATE OF R2 FIRST PART
\subsection{Estimating the individual terms of $\mathcal{R}_2$ (for $n$ fixed)}\label{R2.fixed.n}

We will estimate the measure of each of the summands comprising $\mathcal{R}_2$ individually with the help of the Young tower. Fix $n$ and for the sake of simplicity we will denote
$\ball \cap T^{-n} \ball$ by $\mathcal{B}_n$. Then
\begin{equation} \label{R2 summand}
\mu(\ball \cap T^{-n} \ball) = \mu(\mathcal{B}_n)
= \sum_{i=1}^{\infty} \sum_{j=0}^{R_i-1} m(T^{-j}\mathcal{B}_n \cap \Lambda_i).
\end{equation}
With $s=\floor{n^\eta}$ let $\tilde\Lambda_i$ be as in
Section~\ref{approx_ball}~(II), then
\begin{eqnarray}
\mu(\mathcal{B}_n)
&=& \sum_{i} \sum_{j=0}^{R_i-1} m(T^{-j}\mathcal{B}_n \cap \tilde{\Lambda}_i)
 + \sum_{i} \sum_{j=0}^{R_i-1} m(T^{-j}\mathcal{B}_n \cap (\Lambda_i \setminus \tilde{\Lambda}_i))\notag\\
&\leq& \sum_{i} \sum_{j=0}^{R_i-1} m(T^{-j}\mathcal{B}_n \cap \tilde{\Lambda}_i)
+ \sqrt{n+2}\,\Omega(s) \mu(\ball)\label{R2 summand ctd1}
\end{eqnarray}
using Lemma~\ref{2tallTowersEst} for the second
 term on the RHS to the complement of the set
$ \mathcal{D}_{n,s}$.
Since ($\tilde\zeta$ as in Section~\ref{approx_ball})
$$
m(T^{-j}\mathcal{B}_n \cap \tilde{\zeta}_{\tau}) \;
\leq \sum_{\substack{\tau' \, | \, \exists \, \tau \subset \tau' }}
m(T^{-j} \mathcal{B}_n \cap \tilde{\zeta}_{\tau'})
= \hspace{-0.4cm} \sum_{\substack{\tau' \, | \, \exists \, \tau \subset \tau', \, \tau\in I  \\ T^{-j}\ball \, \cap \, {\zeta}_{\tau'} \ne \varnothing}} \hspace{-0.4cm}
m(T^{-j}\mathcal{B}_n \cap \tilde{\zeta}_{\tau'})
$$
we get, since by~\eqref{long.cylinders} $\bigcup_i\tilde\Lambda_i\subset\bigcup_{\tau\in I}\tilde\zeta_\tau$,
 for each of the summands in the principal term on the RHS of~\eqref{R2 summand ctd1}
\begin{eqnarray*}
m_{\gamma^u}(T^{-j}\mathcal{B}_n \cap \tilde{\Lambda}_i)
&\le& \sum_{\tau \in I_{i, j, n}} m_{\gamma^u}(T^{-j}\mathcal{B}_n \cap \tilde{\zeta}_{\tau})\\
&\le&\sum_{\substack{\tau' \, | \, \exists \, \tau \subset \tau', \tau \in I\\
T^{-j}\ball \, \cap \, {\zeta}_{\tau'} \ne \varnothing}}
\hspace{-0.4cm} m_{\gamma^u}(T^{-j}\mathcal{B}_n \cap \tilde{\zeta}_{\tau'}) \;  \\
&=&\sum_{\substack{\tau' \, | \, \exists \, \tau \subset \tau', \tau \in I \\
T^{-j}\ball \, \cap \, {\zeta}_{\tau'} \ne \varnothing}}
\hspace{-0.4cm} \frac{m_{\gamma^u}(T^{-j}\mathcal{B}_n \cap \tilde{\zeta}_{\tau'})}
{m_{\gamma^u}(\tilde{\zeta}_{\tau'})} \,
m_{\gamma^u}(\tilde{\zeta}_{\tau'})\\
& \leq& C_6 \,\mu(\ball) \hspace{-2mm} \sum_{\substack{\tau' \, | \, \exists \, \tau \subset \tau', \tau \in I \\ T^{-j}\ball \, \cap \, {\zeta}_{\tau'} \ne \varnothing}} \hspace{-0.4cm} 
m_{\gamma^u}(\tilde{\zeta}_{\tau'})
\end{eqnarray*}
where we used Lemma~\ref{case1} in the last step.
From Lemma~\ref{sumCylind} we obtain a bound for the sum of measures of the cylinders, whence
$$
\sum_{\tau \in I} m(T^{-j}\mathcal{B}_n \cap \tilde{\zeta}_{\tau}) \leq C_6 \,\mu(\ball)
 \, m(\gamma^s(T^{-j} (B_{\rho + C_7 \, \alpha^{n/s}})) \cap \Lambda_i)
$$
using the product structure of the measure $m$.
Note that since $B_{\rho + C_7 \, \alpha^{n/s}}(\mathsf{x})
 \subset B_{2\rho}(\mathsf{x}) \cup B_{2C_7 \alpha^{n/s}}(\mathsf{x})$ we obtain
on unstable leaves by Assumption~(A1)(b) that
$$
m_{\gamma^u}(\gamma^s(T^{-j} (B_{\rho + C_7 \, \alpha^{n/s}}))
 \leq m_{\gamma^u}(B_{2\rho}) + m_{\gamma^u}(B_{2C_7 \alpha^{n/s}})
\le(2\rho)^{\varsigma'}+(2C_7 \alpha^{n/s})^{\varsigma'}
$$
 using the geometric regularity (A4)(b)
provided that the radius $2C_7 \, \alpha^{n/s}$ is small enough. Since $n \geq J$ and
 $s$ depends on $n$ we can guarantee the above radius to be sufficiently small provided 
that $\rho$ is small enough. Therefore
$$
m(\gamma^s(T^{-j} (B_{\rho + C_7 \, \alpha^{n/s}})) \cap \Lambda_i)
\le c_1\left((2\rho)^{\varsigma'}+(2C_7 \alpha^{n/s})^{\varsigma'}\right)
$$ 
for some $c_1$.
Thus the first term (principal term) on the RHS of~\eqref{R2 summand ctd1} can be bounded as follows
\begin{align}
\sum_{i} \sum_{j=0}^{R_i-1} m(T^{-j}\mathcal{B}_n \cap \tilde{\Lambda}_i)
&\le \sum_{i} \sum_{j=0}^{R_i-1} \sum_{\tau \in I}
m(T^{-j}\mathcal{B}_n \cap \tilde{\zeta}_{\tau}) \notag\\
&\leq  \sum_{i} \sum_{j=0}^{R_i-1} C_6 \, \mu(\ball) \,
m(\gamma^s(T^{-j} (B_{\rho + C_7 \,\alpha^{n/s}})) \cap \Lambda_i)\notag\\
&\leq C_6'\,\mu(\ball)
\left( \rho^{\varsigma'}+( \alpha^{n/s})^{\varsigma'}\right) \label{R2 summand ctd2}
\end{align}
for a constant $C_6'$.

%%%%%%%%%%%%%%%%%%%%%%%%%%%%%%%%%%%%%%%%%%%%%%
%%%%%%%%%%%%%%%%%%%%%%%%%%%%%%%%%%%%%%%%%%%%%%
%%%%%%%%%%%%    ESTIMATE OF R2 SECOND PART
\subsection{Estimating $\mathcal{R}_2$} \label{est_R2}
Combining inequalities~\eqref{R2 summand ctd1} and~\eqref{R2 summand ctd2} results in
\begin{equation*}
\mu(\ball \cap T^{-n} \ball) = \mu(\mathcal{B}_n)
 \leq C_6'\,\mu(\ball)
\left( \rho^{\varsigma'}+(\alpha^{n/s})^{\varsigma'}\right) + \Omega(s)\sqrt{n+2} \mu(\ball),
\end{equation*}
provided that $n \geq J$ and the center $\mathsf{x}$ of the ball $\ball$ lies outside the set
$ \mathcal{D}_{n,s}$.

As before let $s=\lfloor n^\eta\rfloor$ where
 $\eta\in(\frac3{\lambda-1},1)$.
Summing up the $\mathcal{B}_n$ terms over $n=J,\dots,p-1$, we see that outside the set of forbidden ball centers $\mathcal{V}_\rho \cup\mathcal{X}_\rho$ we get
 with $\tilde{\alpha} = \alpha^{\varsigma'}$
\begin{equation} \label{R2 summand ctd3}
\mathcal{R}_2 = \sum_{n=J}^{p-1} \mu(\ball \cap T^{-n} \ball)
 \leq\mu(\ball) \sum_{n=J}^{p-1} \left(C_6' 
\left( \rho^{\varsigma'}+\tilde\alpha^{n/s}\right) + 2\sqrt{n}\,\Omega(s) \right).
\end{equation}
For $\rho$ small and by Lemma~\ref{omegaEst}
$\Omega(s) \leq C_5 s^{-\theta}$ where $\theta = \frac{\lambda-1}{2}$ we obtain
\begin{eqnarray*}
 \mathcal{R}_2
 &\leq& c_1\mu(\ball)\left( p\,\rho^{\varsigma'}
 +\sum_{n=J}^{p-1} \tilde{\alpha}^{\, n^{1-\eta}}
 +  \sum_{n=J}^{\infty} n^{\frac{1}{2} - \eta \theta}\right)\\
 & \leq &c_2\mu(\ball)\left(p \,\rho^{\varsigma'}
+\tilde{\alpha}^{\frac{1}{2} J^{1-\eta}} + J^{\frac{3}{2} - \eta \theta}\right)
\end{eqnarray*}
since
$\sum_{n=J}^{\infty}\tilde{\alpha}^{\, n^{1-\eta}} \leq c_3 \, \tilde{\alpha}^{\frac{1}{2} J^{1-\eta}} $
for some $c_3$ (and $\rho$ small enough).
As $\tilde{\alpha}^{\frac{1}{2} J^{1-\eta}} \leq J^{- \sigma}$, $\sigma=\eta\theta-\frac32$, for  $\rho$ small, we get
$$
\mathcal{R}_2 \leq C_9 \, \mu(\ball) \left(p\, \rho^{\varsigma'} + J^{-\sigma} \right)
$$
for some  $C_9$.
Note that the above it true provided $\rho$ is sufficiently small and the center $\mathsf{x}$
of the ball $B_\rho$ is not in
$\mathcal{X}_\rho\cup\mathcal{V}_{\rho}$.

%%%%%%%%%%%%%%%%%%%%%%%%%%%%%%%%%%%%%%%%%%
%%%%%%%%%%%%%%%%%%%%%%%%%%%%%%%%%%%%%%%%%%
%%%%%%%%%%%%%%               OPTIMISATION
\subsection{Estimate of the total error} \label{optimization}
Now we want to bound all of the error components from inequality \eqref{errorUSE} with terms of the order of $J^{-\sigma}$ or a negative power of $\abs{\log \rho}$, where $\sigma=\eta\frac{\lambda-1}2-\frac32<\frac{\lambda-4}2$.
To that end we choose the length of the gap $p$ to be
$$
p = \floor{J^{-\sigma} \rho^{-\varsigma'}},
$$
and estimate the three error terms on the RHS of~\eqref{errorUSE} separately.\\
{\bf (I)} The last summand is immediately estimated as
$$
p \mu(\ball)\le p\,c_1\rho^{\varsigma'} \leq  J^{-\sigma}.
$$
{\bf (II)} For the term involving $\mathcal{R}_1$ we obtain
$$
N \mathcal{R}_1 \leq \frac{t}{\mu(\ball)} \, \biggl( \varphi_{p/2} \rho^{-w}
+c_2\,\mu(B_\rho)\left(p^{\frac32-(1-\beta)\frac{\lambda-1}2}+\abs{\log\rho}^{-a}\right) \biggr).
$$
Let  $w<\xi$ and $\varsigma'<\varsigma$, $\hat\varsigma''>\hat\varsigma$ so that 
$\varsigma'(\lambda-1)-\hat\varsigma''-w>0$. Since $\mu(\ball)\ge\rho^{\hat\varsigma''}$
for $\rho$ small enough we  get for the first term on the RHS for $\rho$ small:
$$
\frac{\varphi_{p/2} \rho^{-w}}{\mu(\ball)}
\le c_3\frac{p^{1-\lambda}}{\mu(\ball) \, \rho^w}
\leq c_4 \, \rho^{\varsigma'(\lambda-1)-\hat\varsigma''-w} \, J^{\sigma(\lambda-1)}
\leq J^{-\sigma}
$$
and so (with some $c_5, c_6$) since by Assumption~(A4)(a)
 $\frac32-(1-\beta)\frac{\lambda-1}2<0$ we conclude
$$
N \mathcal{R}_1 \leq c_5\left(J^{-2\sigma} +\abs{\log \rho}^{-a}\right)
\le c_6\abs{\log\rho}^{-\min\{a,\sigma\}}.
$$
{\bf (III)} Utilizing the estimate from  Section~\ref{est_R2} and using the fact that
$N=\lfloor t/\mu(B_\rho)\rfloor$ yield
$$
N \mathcal{R}_2
 \leq t \, C_9 \biggl(p\,\rho^{\varsigma'}+ J^{-\sigma} \biggr)
 \le c_7t\abs{\log\rho}^{-\sigma}
$$
for some $c_7$.

\vspace{0.5cm}
\noindent Combining the results of estimates~(I), (II) and~(III)  above we obtain for $\rho$
sufficiently small the RHS of~\eqref{errorUSE} as follows ($\mathsf{x}\not\in\mathcal{X}_\rho$)
$$
N(\mathcal{R}_1 + \mathcal{R}_2) + p \mu(\ball)
\leq \frac{c_6}{\abs{\log\rho}^{\min\{a,\sigma\}}} +\frac{c_7t}{\abs{\log\rho}^{\sigma}}+ \frac1{J^\sigma}
\le \frac{c_8(1+t)}{\abs{\log\rho}^\kappa}.
$$
for some $c_8$, where  $\kappa = \min\{ \sigma, a\}$ is positive as long as $v<\sigma$.
This now concludes the proof of Theorem~1 as it shows that for $\rho$ small enough and for any
$\mathsf{x} \notin \mathcal{X}_\rho$ we have for $k \leq N$
$$
\Abs{\mathbb{P}(S = k) - \frac{t^k}{k!} \, e^{-t}} \, \leq \, C(1+t^2)\abs{\log \rho}^{-\kappa}.
$$
By~\eqref{yay!_k>N} this estimate also extends to  $k>N$.
By Proposition~\ref{forbidden.set} the size of the forbidden set is then
$\mu(\mathcal{X}_\rho)=\mathcal{O}(\abs{\log\rho}^{-\kappa'})$
where $\kappa'=\hat\sigma$.
The proof of Theorem~1 is thus complete.
\qed

\vspace{3mm}

\noindent {\bf Remark.} In the case when $a\ge\sigma$ and $\beta\le\frac12$ then
$\kappa=\kappa'=\frac{\lambda-4}2$.

%%%%%%%%%%%%%%%%%%%%%%%%%%%%%%%%%%%%%%%%%%%%%%%%
%%%%%%%%%%%%%%%%%%%%%%%%%%%%%%%%%%%%%%%%%%%%%%%%
%%%%%%%%%%%%%%%%%% SECTION VERY SHORT RETURNS
\section{Very Short Returns and Proof of Theorem~2}\label{VeryShortReturns}

In this section we prove Theorem~2 which is a direct consequence of Theorem~1 and
Proposition~\ref{prop.short.returns} below,
which estimates the measure of the set with short return times.
We first state the precise assumptions. Again we use the Young tower construction.

\subsection{Assumptions} \label{ass_T2}

Let $(M, T)$ be a dynamical system equipped with a metric $d$. Assume that the map $T: M \maps M$ is a $C^2$-diffeomorphism. As at the start of the paper the set $\mathcal{V}_\rho \subset M$ is given by
\begin{equation*}
\mathcal{V}_\rho = \{\mathsf{x} \in M: \ball(\mathsf{x}) \cap T^{n}\ball(\mathsf{x}) \ne \varnothing \text{ for some } 1 \leq n < J\},
\end{equation*}
where $J = \floor{\mathfrak{a} \, \abs{\log \rho}}$ and $\mathfrak{a} = (4 \log A)^{-1}$ with
$$
A = \norm{DT}_{\mathscr{L}^\infty} + \norm{DT^{-1}}_{\mathscr{L}^\infty}
$$
($A\ge2$).
Suppose the system can be modeled by a Young tower possessing a reference measure $m$ and that the greatest common divisor of the return times $R_i$ is equal to one.
 Let $\mu$ be the SRB measure associated to the system. We require the following in Proposition~\ref{prop.short.returns}:

\vspace{3mm}

%%%%%%%%%%%%%%%%%%%%%%%%%%%%%%%%%%%%%%%%%%
%%%%%%%%%%%%%%%%%%%%%%%%%%%%%%%%%%%%%%%%%%
%%%%%%%%%%%      ASSUMPTIONS BBBBBBBBBBBBBBBBBB

%%%%%%%%%%%%%%%%%%%%%%%%%%%%
\noindent (B1) {\em Regularity of the Jacobian and the metric on the Tower}\\
The same as Assumption~(A1)

\vspace{3mm}
%%%%%%%%%%%%%%%%%%%%%% B2
\noindent (B2) {\em Polynomial Decay of the Tail} \\
There exist constants $C_1$ and $\lambda>9$ such that
$$
\hat{m}_{\gamma^u}(R>k) \leq C_1 \, k^{-\lambda}
$$
for unstable leaves $\gamma^u$.

\vspace{3mm}

\noindent (B3) {\em Geometric regularity of the measure}\\%%%%%%%%%%%%%%%%%%%%%%%%%%%%%%%%%%%%%%%%%%%%%%% B3
Let $\varsigma$ be the dimension of the measure $\mu$. Suppose that the positive constant
 $\varsigma' < \varsigma$ is fixed. There exists a constant $C_2>0$ and  a set
 $\mathcal{E}_\rho\subset M$  satisfying $\mu(\mathcal{E}_\rho)\le C_2\abs{\log\rho}^{-\lambda/2}$ such that
$$
\qquad m_{\gamma^u}(\ball(\mathsf{x})) \leq C_2 \, \rho^{\varsigma'}
$$
for all $\mathsf{x}\not\in\mathcal{E}_\rho$ and $\rho$ small.

\vspace{3mm}

\noindent Note that in Proposition~\ref{prop.short.returns} we don't require the measure to be $\xi$-regular.

%%%%%%%%%%%%%%%%%%%%%%%%%%%%%%%%%%%%%%%%%%
%%%%%%%%%%%%%%%%%%%%%%%%%%%%%%%%%%%%%%%%%%
%%%%%%%%%%%     SUBSECTION: VERY SHORT RETURNS ESTIMATE
\subsection{Estimate on the measure of $\mathcal{V}_\rho$} \label{shortsMain}

Before we prove the main result of this section we shall present a lemma which
will be needed in the proof of Proposition~\ref{prop.short.returns}.

\begin{lem} \label{rate_diam}
Let $\hat{C}_1>0$ be a constant and $\hat\alpha\in(0,1)$. Then for all sufficiently small $\rho$
$$
4 \, A^{n+b} \rho + A^{b+j} \hat{C}_1 \hat{\alpha}^{\, \abs{\log \rho}^{3/4}}
\le e^{-\abs{\log \rho}^{1/4}}.
$$
for any $n\le J$ and $b,j\le J^\frac14$ ($J = \floor{\mathfrak{a} \, \abs{\log \rho}}$).
\end{lem}

\begin{proof}
By assumption
$$
A^{n+b} \rho \leq A^{2J}  \rho
\leq 2^{-1}A^{2 \mathfrak{a} \abs{\log \rho}}  \rho = 2^{-1}{\rho}^{1-2 \mathfrak{a} \log A} =2^{-1} {\rho}^{1/2}
$$
and also
$$
A^{b+j}\hat\alpha^{\abs{\log\rho}^\frac34}\le A^{1\mathfrak{a}\abs{\log\rho}^\frac14}\hat\alpha^{\abs{\log\rho}^\frac34}
\le(2\hat{C}_1)^{-1}e^{-\abs{\log\rho}^\frac14}.
$$
Since $\rho^\frac12=e^{-\frac12\abs{\log\rho}}$ the statement of the lemma follows for $\rho$ small enough.
\end{proof}

\vspace{3mm}

\noindent Now we can show that the set of centres where small balls have very short returns is small.
To be precise we have the following result:

%%%%%%%%%%%%%%%%%%%%%%%%%%%%%%%%%%%%%%%%%%%%
%%%%%%%%%%%%%%%%%%%%%%%%%%%%%%%%%%%%%%%%%%%%
%%%%%%%%%%%  PROPOSITION: VERY SHORT RETURNS
\begin{prop}\label{prop.short.returns}
There exist constants $C_{10}>0$ such that for all $\rho$ small enough
$$
\mu(\mathcal{V}_\rho)\le\frac{C_{10}}{\abs{\log\rho}^{\frac{\lambda-9}4}}.
$$
\end{prop}

\begin{proof} We largely follow the proof of Lemma~4.1 of~\cite{CC13}.
Let us note that since $T$ is a diffeomorphism one has
$$
\ball(\mathsf{x}) \cap T^{n}\ball(\mathsf{x}) \ne \varnothing \qquad \iff \qquad \ball(\mathsf{x}) \cap T^{-n}\ball(\mathsf{x}) \ne \varnothing.
$$
We partition $\mathcal{V}_\rho$ into level sets $\mathcal{N}_{\rho}(n)$ as follows
$$
\mathcal{V}_\rho = \{\mathsf{x} \in M: \ball(\mathsf{x}) \cap T^{-n}\ball(\mathsf{x}) \ne \varnothing \text{ for some } 1 \leq n < J\}
 = \bigcup_{n=1}^{J-1} \mathcal{N}_{\rho}(n)
 $$
 where
 $$
  \mathcal{N}_{\rho}(n) = \{\mathsf{x} \in M: \ball(\mathsf{x}) \cap T^{-n}\ball(\mathsf{x}) \ne \varnothing \}.
$$
The above union is split into two collections $\mathcal{V}_\rho^1 $ and $\mathcal{V}_\rho^2$, where
\begin{equation*}
\mathcal{V}_\rho^1 = \bigcup_{n=1}^{\floor{\mathfrak{b} J}} \mathcal{N}_{\rho}(n) \quad \text{and} \quad \mathcal{V}_\rho^2 = \bigcup_{n=\roof{\mathfrak{b} J}}^{J} \mathcal{N}_{\rho}(n).
\end{equation*}
aand where the constant $\mathfrak{b} \in (0,1)$ will be chosen below.
In order to find the measure of the total set we will estimate the measures of the two parts separately.

\vspace{3mm}

%%%%%%%%%%%%%%%%%%%%%%%%%%%%%%%%%%%%%%%%%%%%%%%%%%%%%%%%%%%%%%%%%%%%%%%%%%%%%%%% Part 1
\noindent {\bf (I) Estimate of $\mathcal{V}_\rho^2$}

\vspace{1mm}

\noindent We will derive a uniform estimate for the measure of the level sets $\mathcal{N}_{\rho}(n)$ when $n > \mathfrak{b} J$
\begin{align}
\mu(\mathcal{N}_{\rho}(n)) & = \sum_{i=1}^{\infty} \sum_{j=0}^{R_i-1} m(T^{-(n+j)}\mathcal{N}_{\rho}(n) \cap \Lambda_i)\notag \\
& = \sum_{i} \sum_{j=0}^{R_i-1}
m(T^{-(n+j)}\mathcal{N}_{\rho}(n) \cap \tilde{\Lambda}_i)+ \sum_{i} \sum_{j=0}^{R_i-1} m(T^{-(n+j)}\mathcal{N}_{\rho}(n) \cap (\Lambda_i \setminus \tilde{\Lambda}_i)) \label{level_set_N}
\end{align}
as $\mu(\mathcal{N}_{\rho}(n)) = \mu(T^{-n}\mathcal{N}_{\rho}(n))$ and where here we put $s=J^\frac14$
which means
$$
\tilde{\Lambda}_i = \{x \in \Lambda_i: \forall l \leq n, \; R(\hat{T}^{l} x) \leq J^{\frac{1}{4}} \}.
$$
On the RHS the first term is the principal term, and the other two terms will be treated as error terms. Let us
first estimate the error term. Note that like in the proof of Lemma~\ref{2tallTowersEst} and assumption~(B2) the second sum on the RHS can be bounded by
\begin{align*}
 \sum_{i} \sum_{j=0}^{R_i-1} m(T^{-(n+j)}\mathcal{N}_{\rho}(n) \cap (\Lambda_i \setminus \tilde{\Lambda}_i))
&\leq \sum_{i} \sum_{j=0}^{R_i-1} m(\Lambda_i \setminus \tilde{\Lambda}_i) \\
& \leq \sum_{i} R_i \, m(\Lambda_i \setminus \tilde{\Lambda}_i) \\
&\leq \; \; J^{\frac{1}{4}} \, (n+1) \, m(R>J^{\frac{1}{4}}) \\
&  \leq \; \; c_1n J^{-\frac{\lambda-1}{4}} .
\end{align*}
As for the first sum (principal term) on the RHS in~\eqref{level_set_N}, we can decompose the beam base
$\tilde{\Lambda}_i$ into cylinder sets, as in Section~\ref{approx_ball}, with the index set $I$ defined just
as in~\eqref{index set}
$$
\tilde{\Lambda}_i \subset \bigsqcup_{\tau \in I_{i, j, n}} \hspace{-0.25cm} \tilde{\zeta}_{\tau}\\[-0.4 cm].
$$
Then
$$
m(T^{-(n+j)}\mathcal{N}_{\rho}(n) \cap \tilde{\Lambda}_i)
\le\sum_{\tau \in I} m(T^{-(n+j)}\mathcal{N}_{\rho}(n) \cap \tilde{\zeta}_{\tau}) \leq \sum_{\substack{\tau' | \exists \tau \in I \\ \tau \subset \tau'}} m(T^{-(n+j)}\mathcal{N}_{\rho}(n) \cap \tilde{\zeta}_{\tau'}).
$$
Incorporating the above estimates and decomposition into \eqref{level_set_N}, we obtain
($c_2= c_1+C_5^2$)
\begin{equation}\label{level_set_N2}
\mu(\mathcal{N}_{\rho}(n))  \leq  \sum_{i} \sum_{j=0}^{R_i-1} \sum_{\tau' | \exists \tau \in I} m(T^{-(n+j)}\mathcal{N}_{\rho}(n) \cap \tilde{\zeta}_{\tau'}) + c_2nJ^{-\frac{\lambda-1}{4}}.
\end{equation}
We will consider each of the measures $m(T^{-(n+j)}\mathcal{N}_{\rho}(n) \cap \tilde{\zeta}_{\tau'})$ separately.
Let $\tau = (i_0, \hdots, i_l)\in I_{i_0,j,n}$ and $\tau' = (i_0, \hdots, i_{l-1})$ as in Section~\ref{approx_ball}.
By distortion of the Jacobian, Lemma~\ref{Distortion}, we obtain for $\tilde\zeta_{\tau'}\not=\varnothing$
(which implies $\tilde\zeta_{\tau'}=\zeta_{\tau'}$):
\begin{align}
m_{\gamma^u}(T^{-(n+j)}\mathcal{N}_{\rho}(n) \cap \tilde{\zeta}_{\tau'})
&= \frac{m_{\gamma^u}(T^{-(n+j)}\mathcal{N}_{\rho}(n) \cap \tilde{\zeta}_{\tau'})}{m_{\gamma^u}(\tilde{\zeta}_{\tau'})} \, m_{\gamma^u}(\tilde{\zeta}_{\tau'})
\notag\\
&\leq C_3  \, \frac{m_{\hat\gamma^u}(\hat{T}^l (T^{-(n+j)}\mathcal{N}_{\rho}(n) \cap \tilde{\zeta}_{\tau'}))}{m_{\hat\gamma^u}(\Lambda)} \,
m_{\gamma^u}(\tilde{\zeta}_{\tau'}),  \label{level_summand}
\end{align}
where, as before, $\hat\gamma^u=\gamma^u(\hat{T}^lx)$ for $x\in\zeta_{\tau'}\cap\gamma^u$.
For the last line compare with~\eqref{comparison}.
We estimate the numerator by finding a bound for the diameter of the set. Let the points $x$ and $z$ in
$T^{-n}\mathcal{N}_{\rho}(n)$ be such that 
$T^{-j} x, T^{-j} z \in T^{-(n+j)}\mathcal{N}_{\rho}(n) \cap \tilde{\zeta}_{\tau'}\cap\gamma^u$
for an unstable leaf $\gamma^u$. From the from $(n,j)$-minimality of $\tau$ ($R^l  \leq n+j < R^{l+1}$) we know that
$$
\hat{T}^l = T^{n+j-b}, \quad \text{where} \quad b = n+j - R^l < R_{i_l},
$$
therefore
$$
d(\hat{T}^l T^{-j} x, \hat{T}^l T^{-j} z) = d(T^{n-b} x, T^{n-b} z).
$$
Incorporating the definition of the constant $A$
$$
d(T^{n-b} x, T^{n-b} z) \leq A^b \, d(T^n x, T^n z).
$$
Note that $T^n x , T^n z \in \mathcal{N}_{\rho}(n)$ and so
$$
d(T^n x, T^n z) \leq d(T^n x, x) + d(x,z) + d(z, T^n z) \leq 4 A^n \rho + d(x,z).
$$
Further,
$$
d(x, z) \leq A^j \, d(T^{-j}x, T^{-j}z).
$$

\noindent Now, $T^{-j}x, T^{-j}z$ are both elements of ${\zeta}_{\tau'}$ and $s({\zeta}_{\tau'}) = l$. Thus
\begin{equation*}
s(T^{-j}x, T^{-j}z) \geq s({\zeta}_{\tau'}) = l.
\end{equation*}
We derive the lower bound on $l$ from $(n,j)$-minimality which yields
\begin{equation*}
n \leq n+j < R^{l+1} \leq (l+1)J^{1/4} \quad \Longrightarrow \quad l >nJ^{-1/4} - 1.
\end{equation*}
Then employing (A1)(b) and keeping in mind that $n > \mathfrak{b} J$
$$
d(T^{-j}x, T^{-j}z) \leq C_0 \, \alpha^{s(T^{-j}x, T^{-j}z)}
\leq (C_0 / \alpha) \, \alpha^{nJ^{-1/4}}
 \leq (C_0 / \alpha) \, \alpha^{\mathfrak{b} JJ^{-1/4}}
  \leq (C_0 / \alpha) \, \hat{\alpha}^{\abs{\log \rho}^{\, 3/4}},
$$
where $\hat{\alpha} = \alpha^{2^{-3/4} \, {\mathfrak{a}}^{3/4} \, \mathfrak{b}} < 1$.
Therefore
\begin{align*}
d(\hat{T}^l T^{-j} x, \hat{T}^l T^{-j} z) &= d(T^{n-b} x, T^{n-b} z) \\
&\leq A^b \, d(T^n x, T^n z) \\
& \leq A^b \, (4A^n \rho + d(x,z)) \\
&\leq 4 \, A^{n+b} \rho + A^{b+j} \, d(T^{-j}x, T^{-j}z) \\
& \leq 4 \, A^{n+b} \rho + A^{b+j} \, (C_0 / \alpha) \, \hat{\alpha}^{\abs{\log \rho}^{3/4}},
\end{align*}
and by Lemma~\ref{rate_diam} ($\hat{C}_1=C_0 / \alpha$) since $\mathfrak{b}J < n \leq J$ and
$b,j < R_{i_l} < J^{\frac{1}{4}}\le J$:
$$
d(\hat{T}^l T^{-j} x, \hat{T}^l T^{-j} z) \leq e^{- \abs{\log \rho}^{1/4}}.
$$
Taking the supremum over all points $x$ and $z$ yields
$$
\abs{\hat{T}^l (T^{-(n+j)}\mathcal{N}_{\rho}(n) \cap {\zeta}_{\tau'}\cap\gamma^u)} 
\leq e^{-\abs{\log \rho}^{1/4}}.
$$
By assumption (B3) on the relationship between the measure and the metric
$$
m_{\hat\gamma^u}(\hat{T}^l (T^{-(n+j)}\mathcal{N}_{\rho}(n) \cap {\zeta}_{\tau'})) \leq C_2e^{-{\varsigma'}\abs{\log \rho}^{1/4}}
$$
for $\varsigma'<\varsigma$, which implies by the product structure of $m$ that 
$$
m(\hat{T}^l (T^{-(n+j)}\mathcal{N}_{\rho}(n) \cap {\zeta}_{\tau'})) 
\leq c_3e^{-{\varsigma'}\abs{\log \rho}^{1/4}}
$$
Incorporating the estimate into~\eqref{level_summand} yields
$$
m(T^{-(n+j)}\mathcal{N}_{\rho}(n) \cap \tilde{\zeta}_{\tau'})
\leq c_4\, e^{-{\varsigma'}\abs{\log \rho}^{1/4}} m({\tilde\zeta}_{\tau'}),
$$
where $c_4\le \frac{2c_2C_3}{m(\Lambda)}$.
Substituting this into estimate~\eqref{level_set_N2}  we see that ($n\le J$)
$$
\mu(\mathcal{N}_{\rho}(n))\le c_4 \, e^{-{\varsigma'}\abs{\log \rho}^{1/4}}
 \sum_{i} \sum_{j=0}^{R_i-1} \hspace{-0.5cm} \sum_{\substack{\tau' | \exists \tau \in I \\ T^{-(n+j)}\mathcal{N}_{\rho}(n) \cap {\zeta}_{\tau'} \ne \varnothing}} \hspace{-0.8cm} m({\tilde\zeta}_{\tau'}) +c_2 J^{-\frac{\lambda-5}{4}}.
$$
Next we have to bound the triple sum on the RHS.
As we showed before all of the $\tilde{\zeta}_{\tau'}$ with $\tau' \, | \, \exists \tau \in I$ are disjoint and are all subsets of $\tilde{\Lambda}_i$, therefore
$$
\sum_{i} \sum_{j=0}^{R_i-1} \hspace{-0.4cm} \sum_{\substack{\tau' | \exists \tau \in I \\ T^{-(n+j)}\mathcal{N}_{\rho}(n) \cap {\zeta}_{\tau'} \ne \varnothing}} \hspace{-0.8cm} m({\tilde\zeta}_{\tau'})
\le
\sum_{i} \sum_{j=0}^{R_i-1} m(\tilde{\Lambda}_i) \, \leq \mu(M)=1
$$
Hence for $n=\roof{\mathfrak{b} J},\dots,J$ we obtain
$$
\mu(\mathcal{N}_{\rho}(n)) \leq c_4 \, e^{-{\varsigma'}\abs{\log \rho}^{1/4}} + c_2 J^{-\frac{\lambda-5}{4}}
\le c_5J^{-\frac{\lambda-5}4}
$$
for some $c_5$ and $\rho$ small enough, and consequently
\begin{equation}\label{Th2refPt2}
\mu(\mathcal{V}_\rho^2)  \leq \sum_{n=\roof{\mathfrak{b} J}}^{J} \mu(\mathcal{N}_{\rho}(n))
 \leq \sum_{n=\roof{\mathfrak{b} J}}^{J} c_5J^{-\frac{\lambda-5}4}
 \leq c_6 \, \abs{\log\rho}^{- \frac{\lambda-9}{4}}
\end{equation}
for some constant $c_6$ (and $\rho$ small enough) as $J=\lfloor\mathfrak{a}\abs{\log\rho}\rfloor$.

\vspace{3mm}

%%%%%%%%%%%%%%%%%%%%%%%%%%%%%%%%%%%%%%%%%%%%%%%%%%%%%%%%%%%%%%%%%%%%%%%%%%%%%%%% Part 1
\noindent {\bf (II) Estimate of $\mathcal{V}_\rho^1$}

\vspace{1mm}

\noindent Here we consider the case $1 \leq n \leq \floor{\mathfrak{b}J}$. Following~\cite{CC13} we put
\begin{equation*}
s_p = 2^p \, \frac{A^{n \, 2^p}-1}{A^n-1}.
\end{equation*}
By~\cite{CC13} Lemma~B.3 one has $\mathcal{N}_{\rho}(n) \subset \mathcal{N}_{s_p \rho}(2^p n)$
for any $p \geq 1$, and in particular for $p(n) = \floor{\lg \mathfrak{b}J - \lg n} + 1$.
Therefore
$$
\bigcup_{n=1}^{\floor{\mathfrak{b}J}} \mathcal{N}_{\rho}(n) \subset \bigcup_{n=1}^{\floor{\mathfrak{b}J}} \mathcal{N}_{s_{p(n)} \rho}(2^{p(n)} n).
$$
Now define
\begin{equation*}
n' = n 2^{p(n)} \qquad \text{ and } \qquad \rho' = s_{p(n)} \rho.
\end{equation*}
A direct computation shows that $1 \leq n \leq \floor{\mathfrak{b}J}$ implies $\roof{\mathfrak{b}J} \leq n' \leq 2 \mathfrak{b}J$ and so
$$
\mathcal{V}_\rho^1 = \bigcup_{n=1}^{\floor{\mathfrak{b}J}} \mathcal{N}_{\rho}(n) \subset \bigcup_{n=1}^{\floor{\mathfrak{b}J}} \mathcal{N}_{s_{p(n)} \rho}(2^{p(n)} n) \subset \bigcup_{n'=\roof{\mathfrak{b}J}}^{2 \mathfrak{b}J} \mathcal{N}_{\rho'}(n').
$$
Therefore to estimate the measure of $\mathcal{V}_\rho^1$ it suffices to find a bound for
$\mathcal{N}_{\rho'}(n')$ when $n' \geq \mathfrak{b}J$. This is accomplished by using
an argument analogous to the first part of the proof. We replace all the $n$ with $n'$
 and $\rho$ with $\rho'$. The cutoff $J^{1/4} = (\floor{\mathfrak{a} \, \abs{\log \rho}})^{1/4}$
 remains unchanged. We get for $\mathfrak{b}< 1/3$
\begin{equation*}
\mu(\mathcal{N}_{\rho'}(n')) \leq c_4\, e^{-{\varsigma'}\abs{\log \rho}^{1/4}}  + c_2 n' J^{-\frac{\lambda-1}{4}}
\le c_7J^{-\frac{\lambda-5}4}
\end{equation*}
and thus obtain an estimate similar to~\eqref{Th2refPt2}:
$$
\mu(\mathcal{V}_\rho^1) \leq \sum_{n'=\roof{\mathfrak{b}J}}^{2 \mathfrak{b}J} \mu(\mathcal{N}_{\rho'}(n'))
 \leq \sum_{n'=\roof{\mathfrak{b}J}}^{J} c_7J^{-\frac{\lambda-5}4}
 \le c_8\abs{\log\rho}^{-\frac{\lambda-9}4}.
$$

\vspace{3mm}

%%%%%%%%%%%%%%%%%%%%%%%%%%%%%%%%%%%%%%%%%%%%%%%%%%%%%%%%%%%%%%%%%%%%%%%%%%%%%%%% Part 1
\noindent {\bf (III) Final estimate}

\vspace{1mm}

\noindent Overall we obtain for all $\rho$ sufficiently small
$$
\mu(\mathcal{V}_\rho) \leq \mu(\mathcal{V}_\rho^1)+\mu(\mathcal{V}_\rho^2)
\le C_{10}\, \abs{\log \rho}^{- \frac{\lambda-9}4},
$$
where $C_{10}= c_6+c_8$.
 \end{proof}

%%%%%%%%%%%%%%%%%%%%%%%%%%%%%%%%%%%%%%%%%%%%%%
%%%%%%%%%%%%%%%%%%%%%%%%%%%%%%%%%%%%%%%%%%%%%%
%%%%%%%%%%%%%%%%%%%%%%%%%%%%%%%%%%%%%%%%%%%%%%
%%%%%%%%%%  SECTION: SRB MEASURE
\section{Recurrence under an Absolutely Continuous Measure}\label{SRBmeasure}
Here we consider measures that are absolutely continuous with respect to Lebesgue measure.

Let $d$ be the metric on $M$ and $T: M \maps M$  a $C^2$-diffeomorphism with attractor
 $\mathscr{A}$. We assume that the system can be modeled by a Young tower possessing
 a reference measure $\hat{m}$, and that the greatest common divisor of the return times
 $R_i$ is equal to one.
Let $\mu$ be the SRB measure on the attractor, that is $\mu_{\gamma^u}$ is absolutely continuous
with respect to Lebesgue $\ell$ on the unstable leaves $\gamma^u$
where its density function $f$ is regular and bounded.
$\mu_{\gamma^u}(F) = \int_{F} f \, d\ell$
for $F$ on an unstable leaf $\gamma^u$. By~\cite{LSY99}
 $\frac{1}{\mathfrak{c}} \leq f(\mathsf{x}) \leq \mathfrak{c}$  for a.e.
 $ \mathsf{x} \in \mathscr{A}$ and some $\mathfrak{c}>1$.
We require Assumptions (A1), (A2) and (A3) to be satisfied with $\lambda$ larger than $9$.
Note that the attractor $\mathscr{A}\subset M$ is given by
$$
\mathscr{A} = \bigcup_{i=1}^{\infty} \bigcup_{j=0}^{R_i-1} T^j \Lambda_i.
$$

%%%%%%%%%%%%%%%%%%%%%%%%%%%%%%%%%%%%%%%%%%%%%
%%%% SUBSECTION: REGULARITY OF SRB MEASURE
\subsection{Regularity of the SRB measure $\mu$}

\noindent Let us note that by~\cite{CC13} Lemma~B.2 there exists a $\varsigma>0$ and a
set $\mathcal{U}'_\rho\subset M$ such that $\mu(\mathcal{U}'_\rho)\le c_1\Omega(\abs{\log\rho})$ for some
$c_1$ and so that $\mu(B_\rho(\mathsf{x}))\le\rho^\varsigma$ for all $\rho\in (0,1]$ and all
$\mathsf{x}\not\in\mathcal{U}_\rho'$. Hence $\mu$ is geometrically regular for $\varsigma$.

The following proposition shows that the SRB measure $\mu$ is $\xi$-regular for
any $\xi>\frac2u(D+1)-1$ where $u\ge1$ is the dimension of the unstable manifolds and $D$
is the dimension of $M$.

%%%%%%%%%%%%%%%%%%%%%%%%%%%%%%%%%%%%%%%%%%%%%%
%%%%%%%%%%%%%%%    LEMMA ANNULUS
\begin{prop}\label{annulus}
Let $u$ be the dimension of the unstable leaves and let $w_0=\frac2u(D+1)-1$.
Then for every $a<\frac{\lambda-2}2$ and $w>w_0$ there exists a constant $C_{11}$ and a
set $\mathcal{U}''_\rho(w)\subset M$ satisfying $\mu(\mathcal{U}''_\rho)=\mathcal{O}(((w-w_0)\abs{\log\rho})^{-a})$
such that
$$
\mu(B_{\rho+\rho^w}(\mathsf{x})\setminus B_\rho(\mathsf{x}))
\le C_{11}\mu(B_\rho(\mathsf{x}))((w-w_0)\abs{\log\rho})^{-a}
$$
for every $\rho>0$ and for every $\mathsf{x}\not\in\mathcal{U}''_\rho$.
\end{prop}

\begin{proof} Put $\mathcal{A}=B_{\rho+\rho^w}(\mathsf{x})\setminus B_\rho(\mathsf{x})$.
 For $s,l>0$ we let  as in Section~\ref{approx_ball} 
$\tilde\Lambda_i=\{x\in\Lambda_i: R(\hat{T}^k)\le s\;\forall k\le l\}$. Then as in~\eqref{R2 summand ctd1}
we use Lemma~\ref{2tallTowersEst} to obtain for $\mathsf{x}\not\in\mathcal{D}_{sl,s}$:
\begin{eqnarray*}
\mu(\mathcal{A})&\le&\sum_{i}\sum_{j=0}^{R_i-1}m(T^{-j}\mathcal{A}\cap\tilde\Lambda_i)
+2\sqrt{sl}\,\Omega(s)\mu(B_\rho)\\
&\le&\sum_{i}\sum_{j=0}^{R_i-1}\sum_{\tau\in I_{i,j,n}}m(T^{-j}\mathcal{A}\cap\tilde\zeta_\tau)
+2\sqrt{sl}\,\Omega(s)\mu(B_\rho)\\
&\le&\sum_{i}\sum_{j=0}^{R_i-1}\sum_{\substack{\tau' \, | \, \exists \, \tau \subset \tau' \\ \tau \in I\\
T^{-j}\mathcal{A} \, \cap \, {\zeta}_{\tau'} \ne \varnothing}}
\hspace{-0.4cm} m(T^{-j}\mathcal{A}\cap\tilde\zeta_{\tau'})
+2\sqrt{sl}\,\Omega(s)\mu(B_\rho)\\
&\le&C_3'\sum_{i}\sum_{j=0}^{R_i-1}\sum_{\substack{\tau' \, | \, \exists \, \tau \subset \tau' \\ \tau \in I\\
T^{-j}\mathcal{A} \, \cap \, {\zeta}_{\tau'} \ne \varnothing}}
\hspace{-0.4cm}
\int m_{\hat\gamma^u}(\hat{T}^l(T^{-j}\mathcal{A}\cap\tilde\zeta_{\tau'}))
m_{\gamma^u}(\tilde\zeta_{\tau'})\,d\nu(\gamma^u)
+2\sqrt{sl}\,\Omega(s)\mu(B_\rho),
\end{eqnarray*}
where $|\tilde\zeta_{\tau'}|\le\alpha^l$. Here we proceeded as in~\eqref{R2 summand ctd2}
 and put $\hat\gamma^u=\gamma^u(\hat{T}^lx)$ for 
some $x\in\zeta_{\tau'}\cap\gamma^u$. As before, the second term on the RHS estimates
the contributions from terms $m(T^{-j}\mathcal{A}\cap (\Lambda_i\setminus\tilde\Lambda_i))$
and from the tall beams where $R_i>s$. As $\hat{T}^l$ is one-to-one on $\tilde\zeta_{\tau'}$ we get
$$
m_{\gamma^u}(\hat{T}^l(T^{-j}\mathcal{A}\cap\tilde\zeta_{\tau'}))
\le m_{\gamma^u}(\hat{T}^lT^{-j}\mathcal{A}\cap\Lambda)
\le c_1(A^{sl}|\mathcal{A}\cap\tilde\zeta_{\tau'}|)^{\varsigma'}
$$
for $\varsigma'<u$,
 where $u$ is the dimension of the unstable manifolds and the diameter is measures inside
the unstable leaf. Here we used that $m$ on the unstable manifolds (of dimension $u$)
is absolutely continuous with respect to the Lebesgue measure and the fact that the map $T$
expands distances on $M$ by at most a factor $A$ at each iteration.
 Since the map $T$ is $C^2$ on $M$
we get that  $\tilde\zeta_{\tau'}\cap\gamma^u$ are on nearly flat segment of an unstable leaf $\gamma^u$,
which allows us to estimate the measure of the intersection
 $\mathcal{A}\cap\tilde\zeta_{\tau'}\cap\gamma^u$.
Since $\mathcal{A}$ is an annulus of thickness $\rho^w$ there exists a constant $c_2$ such that
$m_{\gamma^u}(\mathcal{A}\cap\tilde\zeta_{\tau'})\le c_2\rho^\frac{w+1}2$.
Since $\hat{T}^l\tilde\zeta_{\tau'}=\Lambda$ if
$\tilde\zeta_{\tau'}\not=\varnothing$ (see  Section~\ref{approx_ball}~(II)) we obtain
(for $\varsigma'<u$)
%(note that $\hat{T}^l$ expands by a factor of at most $A^{sl}\le A^{sn}$ on $\tilde\zeta_\tau$ where
%$\tau=\tau_1\cdots\tau_l$ and $l\ge n/s$)
%%
\begin{eqnarray*}
\mu(\mathcal{A})&\le& (A^{sl}c_2\rho^\frac{w+1}2)^{\varsigma'}
\sum_{i}\sum_{j=0}^{R_i-1}\sum_{\substack{\tau' \, | \, \exists \, \tau \subset \tau' \\ \tau \in I\\
T^{-j}\mathcal{A} \, \cap \, {\zeta}_{\tau'} \ne \varnothing}}
m(\tilde\zeta_{\tau'})
+2\sqrt{sl}\,\Omega(s)\mu(B_\rho)\\
&\le& c_3(A^{sl}\rho^\frac{w+1}2)^{\varsigma'}
(\rho+\rho^w+\alpha^\frac{n}s)^{\varsigma'}
+2\sqrt{sl}\,\Omega(s)\mu(B_\rho)
\end{eqnarray*}
Let $w_0=\frac2u(D+1)-1$ and $k(w)=\frac{w-w_0}{2\log A}$ for $w>w_0$.
Then put $s=\roof{k(w)\abs{\log\rho}^{\tilde\eta}}$ and $l=\floor{\abs{\log\rho}^{1-\tilde\eta}}$,
 where $\tilde\eta\in(0,1)$.
Thus $\frac12\abs{\log\rho}\le sl\le2\abs{\log\rho}$ and we obtain
$(A^{sl}\rho^\frac{w+1}2)^{\varsigma'}\le\rho^{D+1}$ provided $w>w_0$ and 
$\varsigma'<u$ close enough to $u$.
 We have
$\rho^{D+1}\le\mu(B_\rho(\mathsf{x}))$ for all $\mathsf{x}\not\in\hat{\mathcal{U}}''_\rho$
where $\hat{\mathcal{U}}''_\rho\subset M$ is a small set whose measure is by~\cite{CC13} Lemma~A.1
 bounded by $\mathcal{O}(\rho)$ for some constant $C_{11}$.
 This implies
\begin{align*}
\mu(\mathcal{A})
&\le c_5\mu(B_\rho)\left(\rho^{\varsigma'}+\alpha^{\varsigma' l}+\sqrt{sl}\;\Omega(s)\right)\\
&\le c_6\mu(B_\rho)\left(\rho^{\varsigma'}+\alpha^{\varsigma'\abs{\log\rho}^{1-\tilde\eta}}+
(k\abs{\log\rho})^{-\tilde\eta\frac{\lambda-1}2+\frac12}(\abs{\log\rho})^{\frac12(1-\tilde\eta)}\right)\\
&\le \frac{C_{11}}{(w-w_0)^a\abs{\log\rho}^a}\mu(B_\rho)
\end{align*}
for a constant $C_{11}$ and for all $\rho$ small enough. Since we can choose $\tilde\eta$ arbitrarily close to $1$,
we obtain any exponent $a<\frac{\lambda-2}2$. This applies for points $\mathsf{x}\not\in\mathcal{U}''_\rho$,
where the measure of the forbidden set $\mathcal{U}''_\rho=\hat{\mathcal{U}}''_\rho\cup\mathcal{D}_{sl,s}$
is bounded by $\mathcal{O}(1)((w-w_0)\abs{\log\rho})^{-a}$.
\end{proof}

%%%%%%%%%%%%%%%%%%%%%%%%%%%%%%%%%%%%%%%%%%%%%
%%%% SUBSECTION: PROOF OF THEOREM 3
\subsection{Proof of Theorem~3}
To prove this result we combine Theorems~1 and ~2. Let $J = \floor{\mathfrak{a} \, \abs{\log \rho}}$ as before, and take $\mathfrak{a} = [4 \log (\norm{DT}_{\mathscr{L}^\infty} + \norm{DT^{-1}}_{\mathscr{L}^\infty})]^{-1}$. Clearly the dimension $\varsigma$ of 
the measure $m_{\gamma^u}$ is equal to the dimension $u$.

In the previous sub-section we showed that $\mu$ is geometric regular and $\xi$-regular for any $\xi>\frac2u(D+1)-1$ provided one restricts to the set $M\setminus\mathcal{U}_\rho$, where $\mathcal{U}_\rho(w)=\mathcal{U}'_\rho\cup\mathcal{U}''_\rho(w)$ and
$\mu(\mathcal{U}_\rho)\le c_1((w-w_0)\abs{\log\rho})^{-a}$ for some $c_1$ and any $a<\frac{\lambda-2}2$.
Let $g(w)=C_{11}(w-w_0)^{-a}$, then $g(w)$ satisfies the summability condition in Assumption~(A4) for any $\beta>\frac1a>\frac2{\lambda-2}$ (i.e.\ we can always choose $\beta<\frac12$), and so ~(A4) is satisfied outside the set $\mathcal{E}_\rho=\bigcup_{n=J}^\infty\mathcal{U}_\rho(n^\beta)$ whose measure is bounded by $\mathcal{O}(\abs{\log\rho}^{-a})$. 
The  measure of the very short return set
\begin{equation*}
\mathcal{V}_\rho = \{\mathsf{x} \in \mathscr{A}: \ball(\mathsf{x}) \cap T^{n}\ball(\mathsf{x}) \ne \varnothing \text{ for some } 1 \leq n < J\}
\end{equation*}
has been estimated in Proposition~\ref{prop.short.returns}.

With the forbidden set $\mathcal{Z}_\rho = \mathcal{X}_\rho \cup \mathcal{V}_\rho\cup\mathcal{E}_\rho$
whose size is bounded by
$$
\mu(\mathcal{Z}_\rho) \leq C'' \, \abs{\log \rho}^{-\frac{\lambda-9}4}
$$
we get that for any point $\mathsf{x} \in \mathscr{A} \setminus \mathcal{Z}_\rho$  the function $S$ counting the number of visits to the ball
$\ball(\mathsf{x})$ can be approximated by a Poissonian, that is
$$
\Abs{\mathbb{P}(S=k) - e^{-t} \frac{t^k}{k!}} \; \leq \; C \, \abs{\log \rho}^ {- \kappa} \qquad \text{ for all } k \in \nats,
$$
for any $\kappa<\frac{\lambda-7}4$ by the remark following the proof of Theorem~1.
Now we set $\hat{C} = \max(C, C'')$.
\qed

%%%%%%%%%%%%%%%%%%%%%%%%%%%%%%%%%%%%%%%%%%%%%%%%%%%%%%%%%%%%%%%%%%%%%%%
%%%%%%%%%%%%%%%%%%%%%%%%%%%%%%%%%%%%%%%%%%%%%%%%%%%%%%%%%%%%%%%%%%%%%%% Approx. Theorem

\section{Poisson Approximation Theorem} \label{poisson}

This section contains the abstract Poisson approximation theorem which establishes the distance
between sums of $\{0,1\}$-valued dependent random variables $X_n$ and a random variable that
is Poisson distributed. It is used in Section~\ref{set_up_T1} in the proof of Theorem~1 and compares
 the number of occurrences in a finite time interval with the number of occurrences in the same interval
 for a Bernoulli process $\{\tilde{X}_n:n\}$.

\begin{thrm}\cite{CC13} \label{helperTheorem}
Let $(X_n)_{n \in \mathbb{N}}$ be a stationary $\{0,1\}$-valued process and $t$ a positive parameter.
 Let $S_a^b = \sum_{n=a}^b X_n$ and define $S:=S_1^N$ for convenience's sake where
 $N=\floor{t/\epsilon}$ and $\epsilon = \mathbb{P}(X_1 = 1)$. Additionally, let $\nu$ be the Poisson distribution measure with mean $t>0$. Finally, assume that $\epsilon < \frac{t}{2}$. Then there exists a constant $C_{12}$ such that for any $E \subset \nats$, and $2 \leq p < N$ we have
\begin{equation*}
\abs{\mathbb{P}(S \in E) - \nu(E)} \leq C_{12} \#\{E \cap [0,N]\} \; (N(\mathcal{R}_1 + \mathcal{R}_2) + p \epsilon)
\end{equation*}
where,
\begin{align*}
\mathcal{R}_1 &= \sup_{\substack {0 <j<N-p \\ 0<q<N-p-j}} \{ \abs{\mathbb{P}(X_1=1 \land S_{p+1}^{N-j}=q) - \epsilon \, \mathbb{P}(S_{p+1}^{N-j}=q)} \} \\
\mathcal{R}_2 &= \sum_{n=2}^p \mathbb{P}(X_1=1 \land X_n=1).
\end{align*}
\end{thrm}

\begin{proof} Let $(\tilde{X}_n)_{n \in \mathbb{N}}$ be a sequence of independent, identically distributed random variables taking values in $\{0,1\}$, constructed so that $\mathbb{P}(\tilde{X}_1=1)=\epsilon$. Further assume that the $\tilde{X}_n$'s are independent of the $X_n$'s. Let $\tilde{S}=\sum_{n=1}^N \tilde{X}_n$. Then
\begin{align*}
\abs{\mathbb{P}(S \in E) - \nu(E)}
& \leq \abs{\mathbb{P}(S \in E) - \mathbb{P}(\tilde{S} \in E)} + \abs{\mathbb{P}(\tilde{S} \in E) - \nu(E)} \\
& \leq \! \! \sum_{k \in E \cap [0,N]} \! \! \abs{\mathbb{P}(S=k)-\mathbb{P}(\tilde{S}=k)} + \sum_{k=0}^{\infty} \, \Abs{\mathbb{P}(\tilde{S}=k) - \frac{t^k}{k!} e^{-t}}
\end{align*}
Thanks to \cite{AGG89} we can bound the second sum using the estimate
\begin{equation} \label{GoldArratia}
\sum_{k=0}^{\infty} \Abs{\mathbb{P}(\tilde{S}=k) - \frac{t^k}{k!} e^{-t}} \leq \frac{2t^2}{N}.
\end{equation}
For summands of the remaining term we utilize the proof of Theorem 2.1 from \cite{CC13}
according to which for every $k \leq N$,
\begin{equation*}
\abs{\mathbb{P}(S=k)-\mathbb{P}(\tilde{S}=k)}
\leq 2N(\mathcal{R}_1+\mathcal{R}_2+p \epsilon^2) + 4p \epsilon.
\end{equation*}
As $N \leq t/\epsilon$ this becomes
\begin{equation} \label{EstSAndTilde}
\abs{\mathbb{P}(S=k)-\mathbb{P}(\tilde{S}=k)}\leq 6 \, t \, (N(\mathcal{R}_1+\mathcal{R}_2) + p \epsilon).
\end{equation}
Combining~\eqref{GoldArratia} and~\eqref{EstSAndTilde} yields
\begin{align*}
\abs{\mathbb{P}(S \in E) - \nu(E)} &\leq \sum_{k \in E \cap [0,N]} \hspace{-0.2cm} \abs{\mathbb{P}(S=k)-\mathbb{P}(\tilde{S}=k)} \, + \, \frac{2t^2}{N} \\
&\leq \sum_{k \in E \cap [0,N]} \hspace{-0.2cm} 6 \, t \, (N (\mathcal{R}_1 + \mathcal{R}_2) + p \epsilon) \, + \, \frac{2t^2}{t / \epsilon -1} \\
&\leq 6 \, t \, \#\{E \cap [0,N]\} \, (N(\mathcal{R}_1+\mathcal{R}_2) + p \epsilon) + 4 \, t\epsilon \\
&\leq C_{12} \#\{E \cap [0,N]\} \; (N(\mathcal{R}_1 + \mathcal{R}_2) + p \epsilon)
\end{align*}
for some $C_{12}<\infty$.
\end{proof}

%%%%%%%%%%%%%%%%%%%%%%%%%%%%%%%%%%%%%%%%%%%%%%%%%%%%%%%%%%%%%%%%%%
%%%%%%%%%%%%%%%%%%%%%%%%%%%%%%%%%%%%%%%%%%%%%%%%%%%%%%%%%%%%%%%%%% Besi

%%%%%%%%%%%%%%%%%%%%%%%%%%%%%%%%%%%%%%%%%%%%%%%%%%%%%%%%%%%%%%%%%%%%%%% Lemma 4

%%%%%%%%%%%%%%%%%%%%%%%%%%%%%%%%%%%%%%%%%%%%%%%%%%%%%%%%%
%%%%%%%%%%%%%%%%%%%%%%%%%%%%%%%%%%%%%%%%%%%%%%%%%%%%%%%%% Bibliography

\end{document}